\pdfoutput=1

\documentclass[a4paper, 11pt, abstracton, titlepage, oneside]{scrartcl}
\makeatletter

\usepackage[T1]{fontenc}
\usepackage[utf8]{inputenc}
\usepackage[onehalfspacing]{setspace}
\usepackage[headsepline]{scrlayer-scrpage}
\clearscrheadfoot 
\usepackage{layout}
\usepackage{geometry}
\usepackage{adjustbox}
\usepackage{authblk}
\usepackage[titletoc,title]{appendix}

\usepackage[hyphens]{url}
\usepackage{color}

\usepackage{amsmath,amssymb,amsfonts,amsthm, mathtools, bm}
\usepackage{thmtools, thm-restate}
\usepackage{nicefrac}
\usepackage{array}

\usepackage[normal]{threeparttable}
\usepackage{threeparttablex}
\usepackage{tabularx}
\usepackage{longtable}
\usepackage{booktabs}
\usepackage{colortbl}
\usepackage{multirow}
\usepackage{makecell}

\usepackage[acronym]{glossaries}

\usepackage{subcaption}
\usepackage{caption}
\usepackage{floatrow}
\usepackage{graphicx}
\usepackage{placeins}

\usepackage{tikz}
\usetikzlibrary{arrows.meta}
\usetikzlibrary{math}
\usetikzlibrary{shapes}
\usepackage{pgfplots}
\usepgfplotslibrary{groupplots}
\usepgfplotslibrary{dateplot}
\usepackage{tikzscale}

\usepackage{enumerate}
\usepackage{multicol}
\usepackage[author=Patrick]{pdfcomment}
\usepackage{xparse}
\usepackage{twoopt}
\usepackage{etoolbox}

\usepackage{eurosym}
\usepackage{ellipsis}
\usepackage{natbib}
\usepackage{paralist}
\usepackage{ulem}

\usepackage[ruled,linesnumbered]{algorithm2e}
\usepackage{bbm}

\newtheorem{observation}{Observation}
\AtBeginDocument{
	\newgeometry{
		textheight=9in,
		textwidth=6.5in,
		top=1in,
		headheight=14pt,
		headsep=25pt,
		footskip=30pt
	}
}
\newsavebox{\abstractbox}
\renewenvironment{abstract}
{\begin{lrbox}{0}\begin{minipage}{\textwidth}
			\begin{center}\normalfont\sectfont\abstractname\end{center}\quotation}
		{\endquotation\end{minipage}\end{lrbox}%
	\global\setbox\abstractbox=\box0 }

\makeatletter
\expandafter\patchcmd\csname\string\maketitle\endcsname
{\vskip\z@\@plus3fill}
{\vskip\z@\@plus2fill\box\abstractbox\vskip\z@\@plus1fill}
{}{}
\makeatother

\addtokomafont{captionlabel}{\footnotesize\bfseries} 
\addtokomafont{caption}{\footnotesize\bfseries} 

\bibpunct[, ]{(}{)}{,}{a}{}{,}%
\newtheorem{theorem}{Theorem}[section]

\counterwithin*{equation}{section}

\newenvironment{myfont}{\fontfamily{ccr}\selectfont}{\par}
\DeclareTextFontCommand{\textmyfont}{\myfont}

\AtBeginDocument{%
	\abovedisplayskip=7.5pt plus 0pt minus 0pt
	\abovedisplayshortskip=7.5pt plus 0pt minus 0pt
	\belowdisplayskip=7.5pt plus 0pt minus 0pt
	\belowdisplayshortskip=7.5pt plus 0pt minus 0pt
}

\newcolumntype{L}[1]{>{\raggedright\let\newline\\\arraybackslash\hspace{0pt}}p{#1}}
\newcolumntype{C}[1]{>{\centering\let\newline\\\arraybackslash\hspace{0pt}}p{#1}}
\newcolumntype{R}[1]{>{\raggedleft\let\newline\\\arraybackslash\hspace{0pt}}p{#1}}
\renewcommand{\emph}[1]{\textit{#1}}
\graphicspath{{Figures/}}

\begin{document}
\emergencystretch 3em



\newacronym[plural=DASs,firstplural=Demand Adaptive Systems (DASs)]{DAS}{DAS}{Demand Adaptive System}
\newacronym{DRS}{DRS}{Demand Responsive Systems}



\newacronym{HA}{CHA}{Consesus-based heuristic approach}




\newacronym{MDP}{MDP}{Markov decision process}





\newacronym{SDVRP}{SDVRP}{stochastic dynamic vehicle routing problem}
\newacronym{2S-SP}{2S-SP}{Two-stage stochastic program}
\title{\large Operational route planning under uncertainty for Demand Adaptive Systems}

\author[1]{\normalsize Benedikt Lienkamp}
\author[2]{\normalsize Mike Hewitt}
\author[3]{\normalsize Axel Parmentier}
\author[4]{\normalsize Maximilian Schiffer}
\affil{\small 
	TUM School of Management, Technical University of Munich, 80333 Munich, Germany
	
	\scriptsize benedikt.lienkamp@tum.de

    \small
    \textsuperscript{2}Quinlan School of Business, Loyola University Chicago, USA

    \scriptsize mhewitt3@luc.edu

    \small
    \textsuperscript{3}CERMICS, École des Ponts, Marne-la-Vallée, France

    \scriptsize axel.parmentier@enpc.fr
    
	\small
	\textsuperscript{4}TUM School of Management \& Munich Data Science Institute,
	
	Technical University of Munich, 80333 Munich, Germany
	
	\scriptsize schiffer@tum.de}

\date{}

\lehead{\pagemark}
\rohead{\pagemark}

\begin{abstract}
\begin{singlespace}
{\small\noindent With an increasing need for more flexible mobility services, we consider an operational problem arising in the planning of \acrlong{DAS}s. Motivated by the decision of whether to accept or reject passenger requests in real time in a \acrlong{DAS}, we introduce the operational route planning problem of \acrlong{DAS}s. To this end, we propose an algorithmic framework that allows an operator to plan which passengers to serve in a \acrlong{DAS} in real-time. To do so, we model the operational route planning problem as a \acrlong{MDP} and utilize a rolling horizon approach to approximate the \acrlong{MDP} via a two-stage stochastic program in each timestep to decide on the next action. Furthermore, we determine the deterministic equivalent of our approximation through sample-based approximation. This allows us to decompose the deterministic equivalent of our two-stage stochastic program into several full information planning problems, which can be solved in parallel efficiently. Additionally, we propose a consensus-based heuristic and a myopic approach. 
We perform extensive computational experiments based on real-world data provided to us by the public transportation provider of Munich, Germany. We show that our exact decomposition yields the best results in under five seconds, and our heuristic approach reduces the serial computation time by 17 - 57\% compared to our exact decomposition, with a solution quality
decline of less than one percent.
From a managerial perspective, we show that by switching a fixed-line bus route to a \acrlong{DAS}, an operator can increase profit by up to 49\% and the number of served passengers by up to 35\% while only increasing the travel distance of the bus by 14\%.
Furthermore, we show that an operator can reduce their cost per passenger by 43 - 51\% by increasing route flexibility and that incentivizing passengers to walk slightly longer distances reduces the cost per passenger by 83-85\%.

}
\end{singlespace}
\end{abstract}

\maketitle
\setcounter{equation}{0}
\section{Introduction}
One of the development goals outlined in the United Nations' Agenda 2030 for cities \citep{UN2030} is to preserve sustainable, affordable, and accessible mobility. 
Public transportation systems are essential for achieving this goal as they provide access to essential services, education, and employment opportunities for individuals with varying socio-economic backgrounds. Public transportation systems enhance land use efficiency and mitigate urban sprawl, fostering the development of interconnected and compact communities, and reinforcing social bonds. Building and reinforcing extensive public transportation networks enables cities to enhance organizational efficiency, increase ecological sustainability, and broaden societal inclusivity.

In a dense urban area, traditional fixed-line public transportation systems typically have high utilization, particularly during daytime hours, due to a concentrated population that does not have its own personal transportation. 
In contrast, in sparse rural areas, these fixed-line systems, usually operated by busses, exhibit lower utilization in general, reflecting the reduced population density and corresponding lower demand for such services \citep{VDV2020}. This contrast between urban and rural areas shows that the efficiency of public bus transportation systems with fixed routes and schedules depends on the respective operating environment. 

Demand Responsive Systems, originally introduced as Dial-a-Ride, offered door-to-door services, especially catering to users with reduced mobility (\citealp{wilson1971scheduling}; \citealp{ioachim1995request}; \citealp{toth1996fast}). These systems constitute an opportunity to increase efficiency in low-demand scenarios because they do not operate on fixed bus routes and timetables. Accordingly, they provide increased flexibility and can adjust to varying individual transportation needs, allowing for more personalized services while still enabling resource sharing and, thus, a high utilization. However, a drawback of Dial-a-Ride lies in the higher operational costs associated with providing customized services and accommodating specific travel requests \citep{ho2018survey}. Moreover, the real-time nature of pickups and drop-offs, involving significant route changes during operation, can result in unpredictable travel times for passengers, making it less suitable for individuals with time-sensitive commitments.

\glspl{DAS}, as proposed by \cite{malucelli1999demand} and \cite{quadrifoglio2007insertion}, integrate Demand Responsive Systems into the framework of scheduled bus transportation. In the context of \glspl{DAS}, a designated bus line follows a conventional transit route with predetermined compulsory stops, adhering to a fixed schedule with specified time windows for departures at those compulsory stops. In addition to these compulsory stops, passengers have the flexibility to make requests at predefined non-compulsory stops, prompting the system to accommodate detours in the vehicle's route (see Figure~\ref{fig: DAS line}). In the following, we refer to these non-compulsory stops as optional stops.

\begin{figure}[t]
\centering
\begin{subfigure}{0.45\textwidth}
  \centering
    \includegraphics[width=\columnwidth]{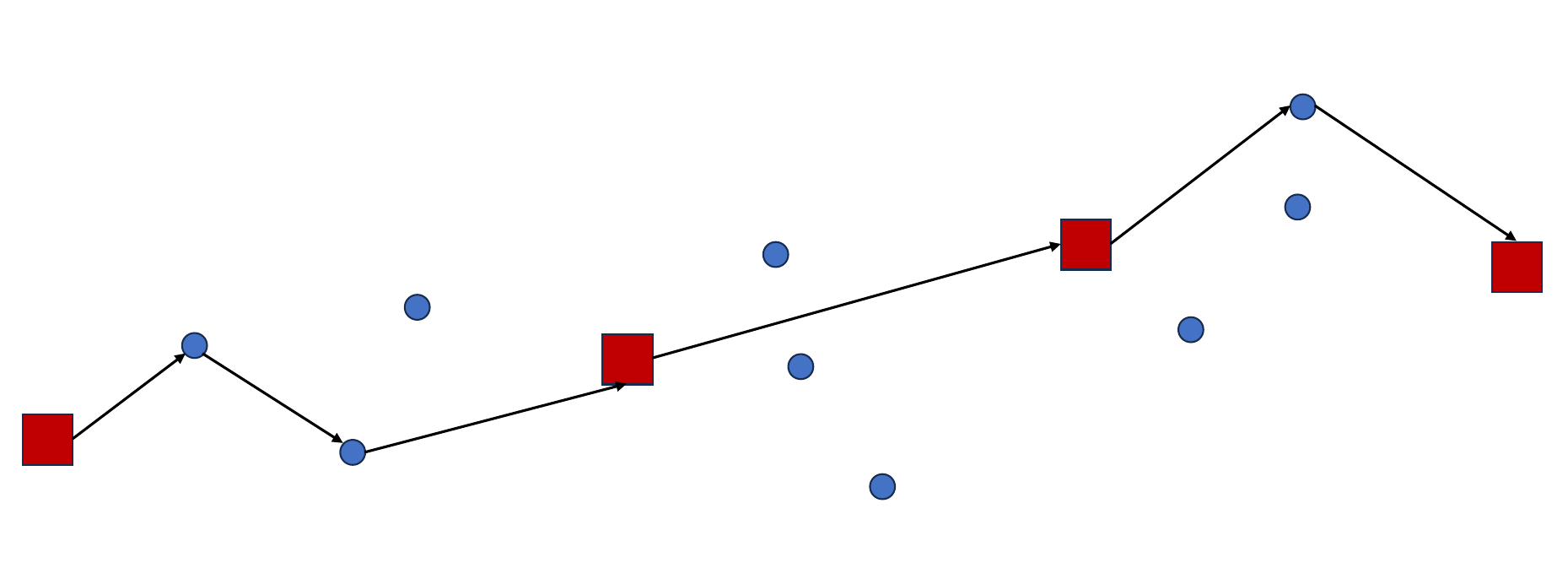}
    \captionsetup{format=hang}
	\caption{Route variant 1}
\end{subfigure}
\begin{subfigure}{0.45\textwidth}
  \centering
    \includegraphics[width=\columnwidth]{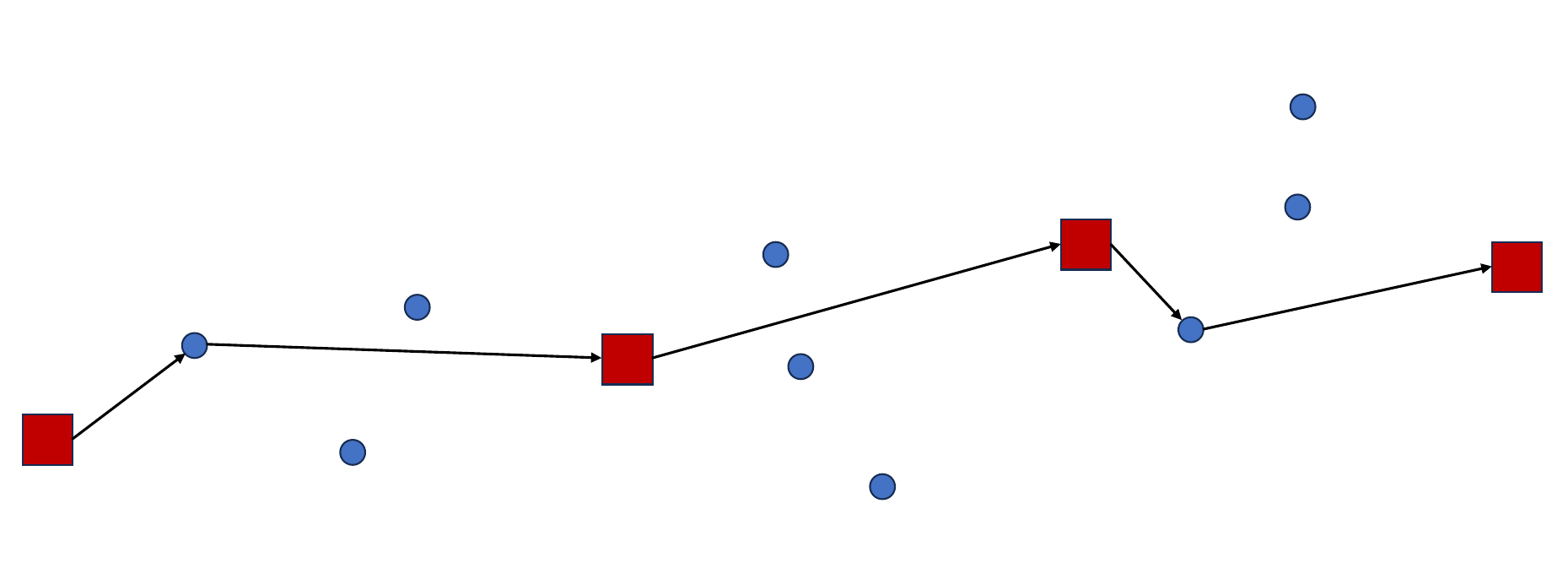}
    \captionsetup{format=hang}
    \caption{Route variant 2}
\end{subfigure}
\caption{A \gls{DAS} line with compulsory stops (red) and optional stops (blue)}
\label{fig: DAS line}
\end{figure}

The fundamental premise of \glspl{DAS} is that the consistency and predictability of public transportation are necessary features. These attributes enable passengers to plan their journeys, allow integration with other modes of transportation, and provide accessibility to the service without necessitating advanced bookings as is the case in Demand Responsive Systems. Here, the reliability inherent in a regular schedule is deemed a crucial asset for public transit, enhancing the overall user experience and optimizing the utilization of transportation infrastructure. Given that premise, a \gls{DAS} introduces a more personalized transportation system, capable of serving a larger proportion of passengers compared to traditional fixed bus lines, all at a reduced cost compared to on-demand door-to-door services \citep{errico2013survey}. This combination of flexibility and predictability effectively combines the strengths of conventional fixed bus lines and Dial-a-Ride, mitigating their respective limitations.

Effectively planning the operations of a \gls{DAS} is a challenging task, given the need to find a balance between the fixed and flexible attributes of traditional and on-demand systems. Additionally, the long-term viability of many public transportation systems depends at least in part on their revenue \citep{farelook}. This planning process entails two phases: a tactical service design phase, wherein bus routes with compulsory and optional stops, as well as schedules, are initially established, and a subsequent operational phase, wherein continuous adjustments to bus routes and schedules in response to user requests are performed. The complexities associated with planning \glspl{DAS} operations have been investigated in existing literature, and \cite{errico2013survey} provide a comprehensive review of the methodological aspects involved in this task.

This paper focuses on the operational phase. More precisely, the problem of determining in real time whether to accept or reject passenger requests of a \gls{DAS} to maximize the operators' profit, which is pivotal for operating a \gls{DAS} economically. This problem setting is similar to a dynamic vehicle routing problem \citep{pillac2013review} with the distinction, that we consider part of the route, i.e., the compulsory stops, to be fixed with a given schedule. \cite{crainic2005meta} analyze this problem under the assumption that the operator has full information, i.e., all passenger requests are known a priori to the acceptance or rejection decision, and develop and compare different meta-heuristics to solve that problem. As this assumption entails long response times towards the passengers requesting service, we explicitly consider that passenger requests occur over time and are not known in advance. This leads to a dynamic and stochastic optimization problem and, thus, short response times whether a passenger request is accepted or rejected.

In conclusion, the planning process of the operations of semi-flexible transit systems and \glspl{DAS} has been studied from tactical and operational perspectives under several system assumptions utilizing simulations, continuous approximations, and combinatorial optimization. 
So far, operational \gls{DAS} problems have only been considered under full demand information in \cite{crainic2005meta}, which does not allow an operator to immediately give feedback to a passenger about the acceptance or rejection decision. Relatedly, the literature on dynamic vehicle routing problems has not considered the setting in which customers are inserted into a pre-determined route and schedule.

To close the research gap outlined above, we propose an algorithmic framework to solve the operational route planning problem of a \gls{DAS}. To this end, we model the operational route planning problem under uncertainty as a \gls{MDP} and utilize a rolling horizon approach to approximate the \gls{MDP} via a two-stage stochastic program in each timestep to decide whether to accept or reject a passenger. We leverage a sample-based approximation of that stochastic program to formulate multiple deterministic problems that can be solved in parallel. Additionally, we propose a heuristic that leverages a consensus-based metric similar to \cite{pisinger2024consensus} to determine which requests to accept.

We utilize this algorithmic framework for a real-world case study in the city of Munich and show that our parallelized decomposition approach outperforms our heuristic and myopic approaches in terms of solution quality. 
Additionally, we see that our exact decomposition yields the best results in under five seconds, and our heuristic approach can reduce the serial computation time by half compared to our exact decomposition, with a solution quality decline of less than one percent. 
We also show that even considering a limited number of future demand scenarios suffices to obtain a good solution quality, especially in high flexibility scenarios. 
From a managerial perspective, we show that by switching a fixed-line bus route to a \gls{DAS}, an operator can increase revenue by up to 49\% and the number of served passengers by up to 35\% while only increasing the travel distance of the bus by 14\%.
Lastly, we show that even though an operator can reduce their cost per passenger by 43 - 51\% by increasing route flexibility, incentivizing passengers to walk 250 meters instead of only 100 meters to their boarding and alighting stops reduces this cost by 83-85\%.

\textbf{Contribution:}
Our contributions are threefold. First, we present a new model of a real-time operational planning problem for \gls{DAS}s that explicitly recognizes their dynamic and stochastic nature. Second, we propose multiple algorithms for solving instances of this model that produce high-quality solutions with some yielding solutions in just seconds with only a small degradation in solution quality. Third, we study the performance of a \gls{DAS}, as compared to a fixed-line bus route, in multiple settings to illustrate the impacts of such systems on both transit system health and passenger experience. 

\textbf{Organization:}
The remainder of this paper is as follows. Section \ref{sec: literature} reviews related literature. We specify our problem setting in Section~\ref{sec: Problem Setting} and develop our methodology in Section~\ref{sec: Methodology}. In Section~\ref{sec: Results}, we describe our instance generation and present our numerical studies. Section~\ref{sec: Conclusion} concludes this paper by summarizing its main findings. 
\section{Literature Review}
\label{sec: literature}
Semi-flexible transit systems, which include \gls{DAS}, have been studied from several perspectives under multiple different assumptions. \cite{errico2013survey} categorize the planning decisions necessary to operate a semi-flexible transit system into strategic, tactical, and operational. 

Strategic planning decisions focus on the definition of the service area, the itineraries in terms of potential compulsory and optional stops, the level of service, i.e., frequencies and fleet size, and whether to use fixed bus routes, \gls{DAS}s, or another type of public transportation system. 
In this context, \cite{daganzo1984checkpoint} compares the cost between fixed route transit and two Dial-a-Ride variants to determine which type of transportation system is better suited to serve a given service area. \cite{quadrifoglio2009methodology} analyze the critical demand density below which their demand-responsive policy outperforms a fixed-route policy, and \cite{li2009optimal} develop an analytical model which determines the optimal number of independently operated zones, the service area should be divided into, to balance customer service quality and vehicle operating costs. \cite{pratelli2001mathematical} propose a mathematical formulation to define the service area of a semi-flexible line based on the total aggregated travel time of passengers spent on board during deviations from the original main route, the walking times, and the increase of waiting times at compulsory stops. \cite{qiu2015demi} study a so-called demi-flexible operating policy of a feeder system in Zhengzhou City (China). This feeder system can deviate from its fixed route to offer door-to-door service. The demi-flexible operating policy allows rejected passengers to utilize curb-to-curb stops of accepted passengers to carry out their trips. The authors show that under certain assumptions, this demi-flexible policy offers advantages in terms of affordability and efficiency compared to fixed-route and fully flexible route policies. \cite{viergutz2019demand} perform a simulation-based study in a rural area in Germany and show that stop-based and door-to-door \gls{DRS} are uneconomical in their rural service area because of their high operational costs. \cite{bischoff2019impact} perform a simulation-based study in the city of Cottbus (Germany) to determine the fleet size of a stop-based and a door-to-door \gls{DRS} and show that for autonomous vehicles, each of the two \gls{DRS} services is significantly cheaper during peak hours compared to fixed-line policies.

Tactical planning decisions focus on (i) which compulsory and optional stops to include in the design of a semi-flexible line, (ii) in which sequence the compulsory stops should be visited, (iii) the tour segments, i.e., between which two compulsory stops can an optional stop be visited, and (iv) the design of a so-called master schedule, which determines the time-windows of the compulsory stops. \cite{errico2013survey} categorize tactical planning according to the used solution method into continuous approximations and combinatorial optimization. As our methodology belongs to the later solution method, we only review combinatorial optimization approaches for \gls{DAS} in this paragraph. \cite{errico2008design} and later \cite{errico2011design} introduce the Single-line \gls{DAS} Design Problem (SDDP), which assumes the set of stops to be served by the \gls{DAS} line, the transportation demand, and the travel time between all stop pairs to be known. Based on this information, the SDDP determines the compulsory stops, their sequence, the \gls{DAS} segments, and the master schedule of the compulsory stops. The authors propose two decompositions of the SDDP consisting of the selection of the compulsory stops, the General Minimum Latency Problem (GMLP), which is a sequencing problem of the compulsory and optional stops, and the master schedule problem. \cite{errico2011design} address the selection of the compulsory stops and compare the two decompositions. \cite{errico2017benders} propose the traveling salesman problem with generalized latency (TSP-GL), a variant of the GMLP where the objective function combines the vehicle routing cost and the passenger's travel times, and develop a Benders' decomposition approach to solve it efficiently. \cite{lienkamp2023branch} build on this methodology and introduce the stochastic TSP-GL, incorporating stochastic passenger demand into the problem setting and solve it via a Branch-and-Price approach. They show that an operator can derive more robust solutions by considering stochastic passenger demand than in a deterministic setting. \cite{crainic2012designing} proposes a mathematical description and a solution method based on probabilistic approximations to solve the master schedule problem.
\newline
Lastly, operational planning determines vehicle and driver schedules as well as real-time adjustments to the \gls{DAS} route. There are two main operational problem categories: dynamic online and dynamic offline problems. \cite{vansteenwegen2022survey} give an extensive overview of this categorization. 
Dynamic online problems assume that requests can occur during the planning horizon and after the bus starts its trip. Here, \cite{pei2019operational} utilize a tabu search approach to reduce the number of visited stops in a fixed-bus line, maximizing the total system income, which is the income minus operating and passenger time costs. They show that during off-peak hours, their approach outperforms fixed-line operations. \cite{quadrifoglio2007insertion} propose an insertion heuristic for Mobility Allowance Shuttle
Transit (MAST), a special case of semi-flexible transit systems, which is described in detail in \cite{errico2013survey}. They show that their approach can be used to automate a MAST in Los Angeles (USA) and that their heuristic results are close to optimal solutions.

Dynamic offline strategies assume that requests can only be issued before the bus starts its trip. In this context, \cite{malucelli1999demand} introduce \gls{DAS} with three operational policies differing in user requirements. They propose linear integer formulations for these policies, which they solve through Lagrangian approaches and column generation, and combine these with heuristic approaches. \cite{malucelli2001adaptive} and \cite{crainic2005meta} introduce a GRASP-based memory-enhanced constructive heuristic and a tabu search meta-heuristic to solve the operational problem of \gls{DAS} under full information. Additionally, they combine the two heuristic approaches, resulting in multiple hybrid approaches, and show that for simple instances, a tabu search-based approach performs best. In contrast, a hybrid approach yields the best results for more complex instances. \cite{qiu2014dynamic} study a MAST system in Los Angeles (USA) and show that their dynamic stop strategy can reduce user cost by up to 30\%. 
We refer to \cite{errico2013survey} and \cite{vansteenwegen2022survey} for an in-depth overview of semi-flexible and demand-responsive systems.

More generally, the \gls{SDVRP} is an extension of the vehicle routing problem where both uncertainty, e.g., random customer demands or travel times, and real-time changes are considered. There are four types of methods commonly applied to solve \gls{SDVRP}s. Reoptimization approaches only use information that is currently available to the operator (\citealp{branchini2009adaptive}; \citealp{gendreau2006neighborhood}). Policy function approximations aim to imitate effective decision-making in practice (\citealp{branke2005waiting}; \citealp{thomas2004anticipatory}). Cost function approximations learn the value of a policy via repeated simulations and training (\citealp{riley2020real}; \citealp{al2020approximate}). Look-ahead algorithms sample stochastic information and derive a policy based on the sampled scenarios (\citealp{schilde2014integrating}; \citealp{azi2012dynamic}). We refer to \cite{soeffker2022stochastic} and \cite{ulmer2020modeling} for an in-depth overview of modeling stochastic dynamic vehicle routing problems. 

In summary, semi-flexible transit systems and \gls{DAS} have been studied from strategic, tactical, and operational perspectives under several system assumptions utilizing simulations, continuous approximations, and combinatorial optimization. 
So far, operational \gls{DAS} route planning problems have only been considered in a static setting under full demand information in \cite{crainic2005meta}, which does not allow an operator to give feedback to a passenger about the acceptance or rejection decision in a dynamic way. Furthermore, stochastic dynamic vehicle routing problems do not consider fixed routes and schedules.

\section{Problem description and mathematical formulations}
\label{sec: Problem Setting}
A \gls{DAS} route consists of \textit{compulsory stops} at which the bus performing the route has to arrive at predefined \textit{time windows}. 
We define a set of \textit{optional stops} for each pair of consecutive compulsory stops that may be served by the vehicle when driving from one compulsory stop to the next. Travel times between all stops and service times at stops are known. Passengers seek transportation between pick-up and drop-off locations. To do so, they issue \textit{requests}, which induce sets of pick-up and drop-off stops within a given walking distance from their corresponding pick-up and drop-off location determined by the operator. These pick-up and drop-off stops are subsets of all compulsory and optional stops. 
If we accept only passenger requests whose pick-up and drop-off stops consist of compulsory stops, the vehicle's route consists only of the compulsory stops, whose order is pre-determined and can be inferred from the time windows at compulsory stops. If we accept a passenger request whose induced pick-up or drop-off stops contain an optional stop, the vehicle may reroute such that the route includes an optional stop. We assume an uncapacitated vehicle, i.e., a passenger request is only rejected if doing so renders visiting a compulsory stop during its time window impossible or the cost of doing so is too high.

Each passenger request, if served, induces a utility that may correspond to the fare charged for the particular request but may also account for service quality or priority indicators, e.g., if a request was already rejected in the previous route \citep{crainic2005meta}. In this work, we consider the utility to be the fare charged. We aim to maximize the profit of our tour, i.e., the difference between the revenue of accepted passenger requests and cost.

To explain the routing aspects of our problem, including the potential and impact of accepting or rejecting passenger requests, we first formulate the full information problem setting, which we will later utilize to decompose our operational route planning problem under uncertainty into multiple full information problems. Second, we introduce the operational route planning problem under uncertainty as an \gls{MDP}.

\subsection{Full information problem setting}
In our full information setting, information about all requests is known before taking any acceptance or rejection decision, and determining the vehicle's route. Similar to \cite{crainic2005meta}, we define the set of all compulsory stops as $H = \{f_1, f_2, ..., f_{n+1}\}$ and assume our route starts at compulsory stop $f_1$ and ends at compulsory stop $f_{n+1}$. Furthermore, there are predefined time windows $[a_h, b_h]$ for each compulsory stop $f_h, h = 1, \ldots, n+1$ at which the vehicle has to serve stop $f_h$. If the vehicle arrives earlier than $a_h$ at $f_h$, the vehicle must wait. We associate a set $F_h$ of optional stops to each pair of consecutive compulsory stops $\langle f_h, f_{h+1} \rangle$. The vehicle only serves an optional stop if we accept a request inducing pick-up or drop off at the stop. Sets $F_h$ are mutually disjoint. We define \textit{segment h} for each pair $\langle f_h, f_{h+1} \rangle$ as the directed graph $G_h = (N_h, A_h)$, where $N_h=F_h \cup \{f_h, f_{h+1}\}$ is the node set, and $A_h \subseteq N_h \times N_h$ is the set of arcs connecting the stops. Then, $G = \bigcup_h G_h$ is the entire graph. Travel costs $c_{ij}$ and travel times $\tau_{ij}$, which include service time at node $i$, for each arc $(i,j) \in A$ are given and positive.

\begin{figure}[!b]
\centering
    \includegraphics[width=0.9\columnwidth]{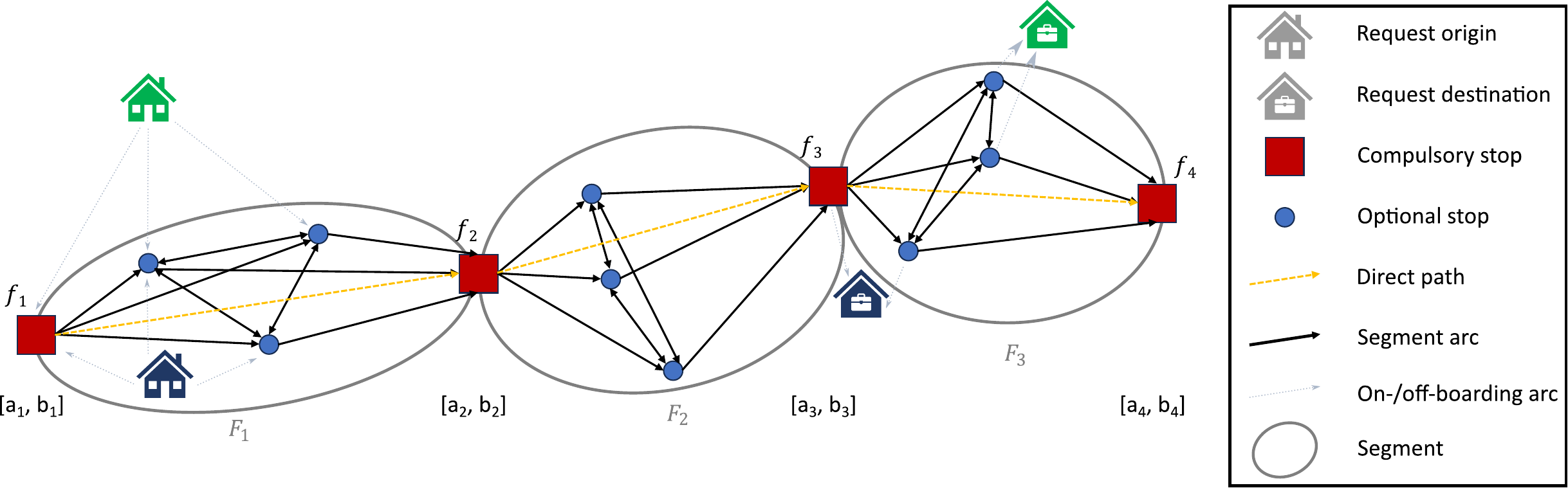}
\caption{Planned bus route with compulsory and optional stops before requests received}
\label{fig: offline problem}
\end{figure}

\begin{figure}[!t]
\centering
\begin{subfigure}{0.7\textwidth}
  \centering
    \includegraphics[width=\columnwidth]{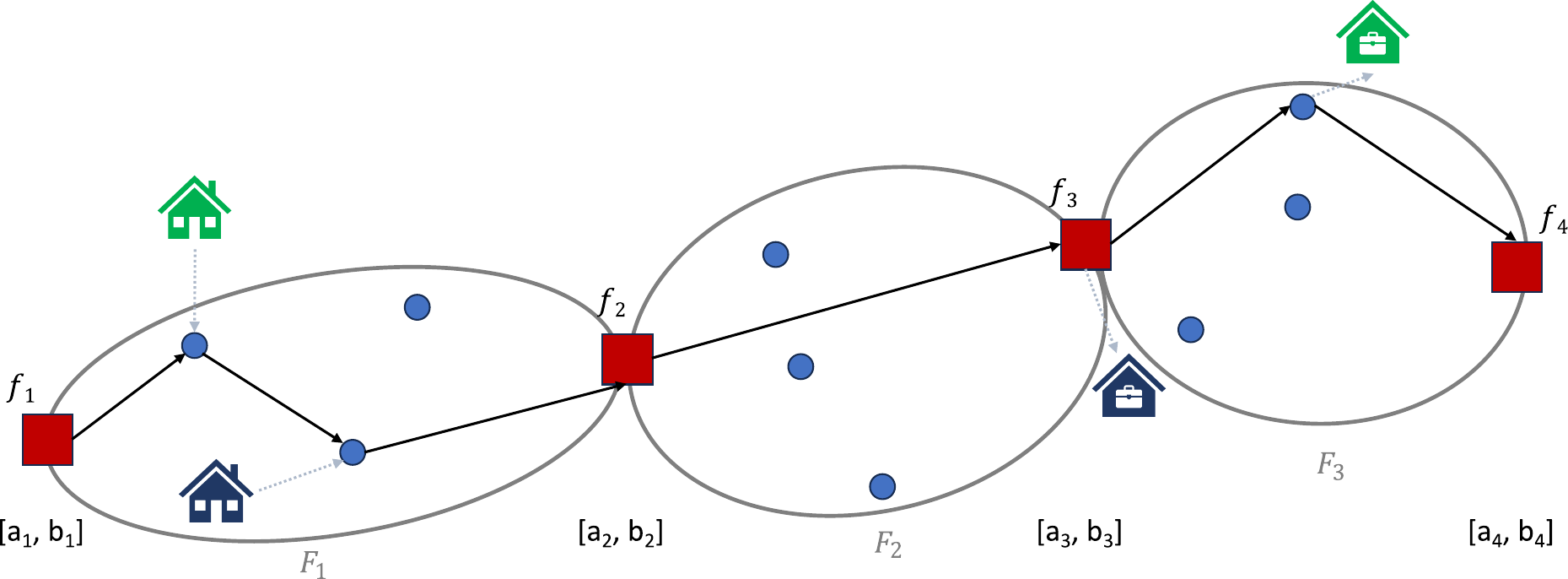}
    \captionsetup{format=hang}
	\caption{Route realization one with two accepted passenger requests}
\end{subfigure}

\begin{subfigure}{0.7\textwidth}
  \centering
    \includegraphics[width=\columnwidth]{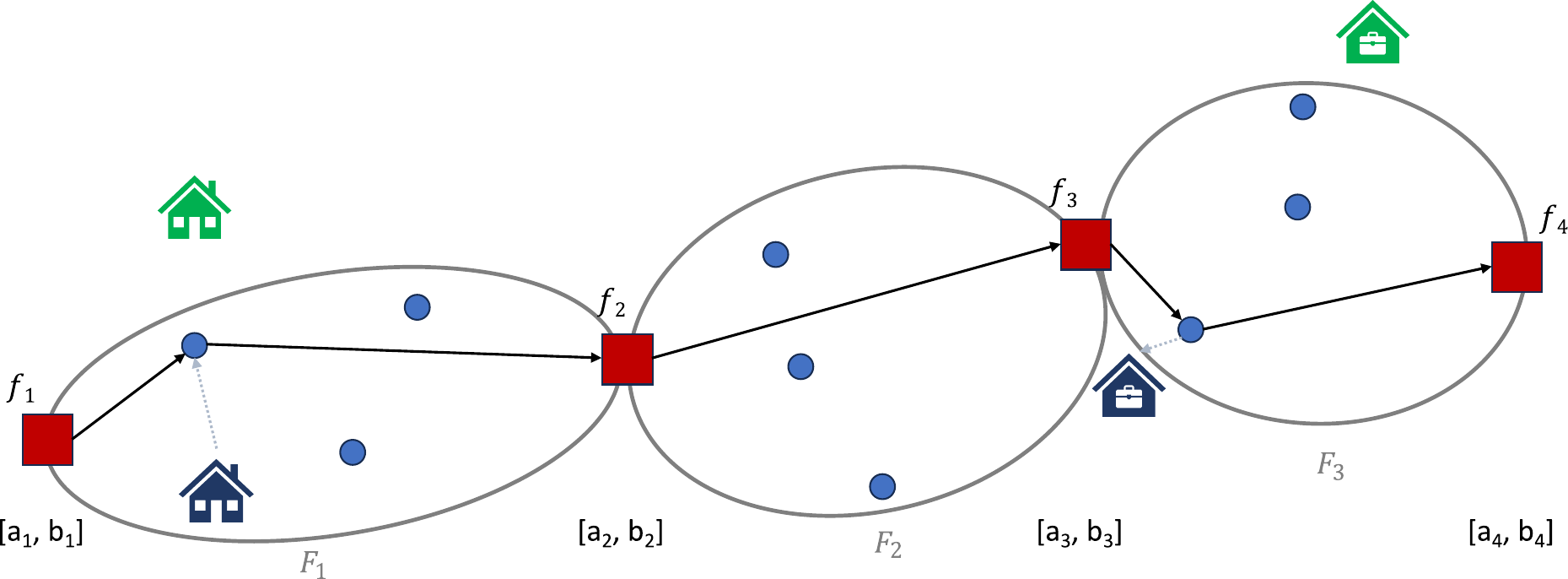}
    \captionsetup{format=hang}
    \caption{Route realization two with one accepted passenger request}
\end{subfigure}
\caption{Different bus routes for different accepted passenger requests}
\label{fig: Route options}
\end{figure}

We denote the \textit{request set} as $R$, where $r \in R$ is defined as a pair $\langle s(r), d(r) \rangle$ of on and off-boarding stops. Here, $s(r), d(r) \subseteq \bigcup_{h} N_h$ are the sets of all stops within a certain proximity of the origin or destination coordinates of request $r \in R$ respectively. For each $r \in R$, let $n_1 \in s(r)$ and $n_2 \in d(r)$ with $n_1$ and $n_2$ being an optional stop and let $h(n_1)$ and $h(n_2)$ denote the on and off-boarding segments of $r$, respectively. We assume that $h(n_1) < h(n_2)$ holds, i.e., $n_1$ and $n_2$ do not belong to the same segment. Accordingly, we do not have to consider precedence constraints between stops within the same segment. Additionally, we implicitly handle precedence constraints regarding on and off-boarding stops of each request by the sequencing of compulsory stops. 

Figure~\ref{fig: offline problem} shows an example of a problem setting with compulsory stops $f_h$ (red), optional stops (blue), time windows $[a_h, b_h]$, segments $F_h$ (grey), segment arcs $A_{h}$ (black), and on-/off-boarding arcs (dotted grey) defining the on and off-boarding stops. Additionally, there exists the direct path only including the compulsory stops (yellow), and two passenger requests represented by houses and work buildings.
Figure~\ref{fig: Route options} shows two possible route realizations of this problem setting. Let $y_r$ be the decision variables that indicate whether we accept request $r\in R$ $(y_r = 1)$ or not $(y_r = 0)$. In the first route realization, both requests are served, i.e., $y_{r_1} = y_{r_2} = 1$. In the second route realization, only one request is served, i.e., $y_{r_1} = 1$ and $y_{r_2} = 0$. Note that even though the second request (green) may enter the bus at the same station as in the first route variant, it can not reach its destination.

We associate a \textit{utility} $u(r) > 0$ with each request $r \in R$, and aim to maximize the profit of our tour, i.e., the difference between the utility of accepted passenger requests and cost. 
In this context, we use decision variables $x_{ij}$ to indicate whether arc $(i,j) \in A$ is used in the vehicle's tour $(x_{ij} = 1)$ or not $(x_{ij} = 0)$. Furthermore, decision variable $t_h$ is the starting time from all compulsory stops $f_h, h = 1, \ldots, n+1$ and $\delta_{s(r)}^i$ ($\delta_{d(r)}^i$) is a constant of value 1 if $i \in s(r)$ ($i \in d(r)$) and 0 otherwise.

We define $d_i$ as

\[
d_i =
\begin{cases}
    1, & \text{if } i = f_1, \\
    -1, & \text{if } i = f_{n+1}, \\
    0, & \text{otherwise}.
\end{cases}
\]

With this notation, we formulate the \gls{DAS} full information problem as follows.  

{
\allowdisplaybreaks
\begin{subequations}
\begin{align}
\max \quad \sum_{r \in R} u(r) y_r - \sum_{i,j \in A} c_{ij} x_{ij} \label{obj: offline arcs}
\end{align}
\begin{align}
y_r & \leq \sum_{i \in N} \sum_{j \in N^+(i)} \delta_{s(r)}^i x_{ij} & \forall r \in R\label{eq: req offline arcs origin}\\
y_r & \leq \sum_{j \in N} \sum_{i \in N^-(j)} \delta_{d(r)}^j x_{ij}  & \forall r \in R \label{eq: req offline arcs destination}\\
\sum_{j \in N^+(i)} x_{ij} - \sum_{j \in N^-(i)} x_{ji} & = d_i  & \forall i \in N \label{eq: flow conservation offline arcs}\\
t_h + \sum_{i,j \in A_h} \tau_{ij} x_{ij} & \leq t_{h+1} &  h=1, \ldots, n \label{eq: node times offline arcs}\\
a_h & \leq t_h \leq b_h &  h=1, \ldots, n+1 \label{eq: time windows offline arcs}\\
\sum_{i \in Q} \sum_{j \neq i, j \in Q} x_{ij} & \leq \vert Q \vert - 1 & \forall Q \subseteq N_h, \vert Q \vert \geq 2, h = 1, \ldots, n \label{eq: subtour elim}\\
y_r & \in \{0,1\} &  \forall r \in R \label{eq: dimension y offline arcs}\\
x_{ij} & \in \{0,1\} &  \forall i,j \in A\label{eq: dimension x offline arcs}
\end{align}
\label{prob: offline arcs}
\end{subequations}
}
Objective function~\eqref{obj: offline arcs} maximizes the operators profit according to the request utility $u(r)$ and the arc costs $c_{ij}$. Constraints~\eqref{eq: req offline arcs origin} and \eqref{eq: req offline arcs destination} ensure that a request $r \in R$ can only be served if one stop out of the set of on-boarding stops $s(r)$ and one stop of the set of off-boarding stops $d(r)$ are visited by the vehicle. Constraints~\eqref{eq: flow conservation offline arcs} ensure flow conservation at all optional and compulsory stops.
Note that we do not need to specify that all compulsory stops are included in the vehicle's tour as the graph structure implies that each compulsory stop $f_h$, $h = 1, \ldots, n$ only has ingoing arcs $N^-(f_{h})$ from segment $h$ and outgoing arcs $N^+(f_{h})$ to segment $f_{h+1}$, i.e., a path from $f_1$ to $f_{n+1}$ must include all compulsory stops.
Constraints~\eqref{eq: node times offline arcs} define the starting times $t_h$ at all compulsory stops. Constraints~\eqref{eq: time windows offline arcs} ensure that the time windows at the compulsory stops are met. 
Constraints~\eqref{eq: subtour elim} ensure that no subtours within the segments exist and are ensured through lazy constraints.
Finally, Constraints~\eqref{eq: dimension y offline arcs} and \eqref{eq: dimension x offline arcs} define the domains of $y_r$ and $x_{ij}$.

\subsection{Dynamic and stochastic problem setting}
\label{subsec: problem setting dynamic}
In the operational route planning problem under uncertainty, passengers issue requests one after another before the vehicle begins its route. We model decisions as being made according to a discrete time horizon $\theta = 1, \ldots, T - 1$, at which one new request $r_{\theta}$ may arrive, and we must decide whether to accept or reject that request. Similar to the full information setting, we assume each request to induce a pair of on and off-boarding stop sets $\langle s(r_{\theta}), d(r_{\theta}) \rangle$ with a utility $u(r_{\theta})$. At decision time $T < a_1$, we have to determine the optimal route serving all accepted requests and respecting all time windows at compulsory stops. 

We model our dynamic operational route planning problem under uncertainty as a finite-horizon MDP with the following state spaces, action spaces, and value functions.

\medskip

\noindent\textbf{System states:} At any event-based decision time $\theta = 1, \ldots, T$, the system state $\chi_{\theta}$ consists of previously accepted passenger requests $S_{\theta}$ that have to be fulfilled, a new passenger request $r_{\theta}$, and decision time $\theta$.
\medskip

\noindent\textbf{Action spaces:}
At decision time $\theta = 1, \ldots, T - 1$, our action space $\mathcal{A}(\chi_{\theta})$ entails either accepting or rejecting the new request $r_{\theta}$. If, due to the time windows at the compulsory stops, there is no route that can serve all previously accepted passengers $\chi_{\theta}$ and the current request, the action space will be reduced to rejecting the new request. This ensures feasibility at decision time $T$.
At decision time $T$, our action space $\mathcal{A}_T(\chi_{\theta})$ is the set of all routes serving all accepted requests $\chi_{\theta}$ and respecting all time windows at compulsory stops.
\medskip

\noindent\textbf{Value functions:} We define our value function at decision times $\theta = 1, \ldots, T - 1$ as follows.
\begin{equation}
    V_{\theta} (\chi_{\theta}) = \max_{a \in \mathcal{A}(\chi_{\theta})} \bigl\{ R(\chi_{\theta}, a) + \mathbb{E} \bigl[ V_{\theta + 1}(X(\chi_{\theta}, a, W_{\theta + 1}))\bigr] \bigl\}
\end{equation}

Here, $R(\chi_{\theta}, a)$ is the payoff from taking action $a \in \mathcal{A}(\chi_{\theta})$, i.e., accepting a new request $r_{\theta}$ yields a reward of $u(r_{\theta})$, and the system state changes from $\chi_{\theta}$ to $\chi_{\theta + 1} = X(\chi_{\theta}, a, W_{\theta + 1})$ through transition function $X(\cdot)$ when action $a$ is taken and the random variable $W_{\theta + 1}$ is realized. Here $W_{\theta + 1}$ is the random variable of a new request $r_{\theta + 1}$, arriving between decision time $\theta$ and $\theta + 1$. 
As we only have to decide the \gls{DAS} route after all passenger request decisions have been taken, our value function at decision time $T$ consists of the cost of serving all accepted passenger requests and respecting the compulsory stop's time windows. Accordingly, our value function at decision time $T$ is as follows.

\begin{equation}
    V_{T} (\chi_T) = \max_{a \in \mathcal{A}_T(\chi_{T})} R_T(a)
\end{equation}

Here, $R_T(a)$ < 0 is the cost of the \gls{DAS} route chosen through action $a$.

\noindent\textbf{Policy:} Let $S_{\Theta}$ denote the set of all possible system states. A policy $\delta : S_{\Theta} \rightarrow \mathcal{A}(\chi_{\Theta})$ is a mapping that assigns to any system state $\chi_{\Theta}$ a decision $a \in \mathcal{A}(\chi_{\Theta})$. 
\medskip

\noindent\textbf{Full information upper bound:} We can compute a full-information upper bound by supposing that all requests $\bigcup_{\theta} r_{\theta}$ are already known at timestep ${\theta} = 1$ and solve this problem via our full information model (Problem~\ref{prob: offline arcs}).
In this context, we define a sample path for a given decision time $\theta$ as one realization of requests at each subsequent decision time $\theta'$ with $\theta < \theta' < T$. We will use these sample paths in the methodology we present next to approximate our dynamic and stochastic problem with one that is stochastic but static as we presume knowledge of the distribution of sample paths.

\section{Methodology}
In this section, we first show how we utilize a rolling horizon look-ahead approach in which the decision of whether to accept or reject a request at a given decision time is made based on an estimate of its impact on both the ability to accept future requests and routing costs. This estimate is based on an approximation in which future requests remain uncertain but are static, as they are revealed simultaneously. To compute this estimate, we formulate a two-stage stochastic program in which the first stage reflects the decision to accept or reject the current request. The second stage problem presumes immediate and complete information for a given sample path and thus knowledge of all future requests that can be accepted or rejected. This second stage problem has an objective that reflects the revenues gained from already accepted requests and accepting future requests in periods before decision time $T$, and routing costs incurred at the last decision time $T$.

We then show how to utilize a sample-based approximation to approximate the expectation of the objective function of the second-stage problem through a multi-scenario deterministic model. Then, we show how to decompose our sample-based approximation into several full information models of the operational route planning problem (Section~\ref{chapter: sample based decomposition}). Afterwards, we propose a consensus-based heuristic approach that reduces the number of full information problems needed to solve the operational route planning problem under uncertainty by 50\% at each decision time $\theta$ (Section~\ref{chapter: heuristic}). Finally, we introduce a straightforward myopic approach (Section~\ref{chapter: greedy}) for benchmarking purposes.
\label{sec: Methodology}

\subsection*{Rolling horizon framework}
\label{sec: rolling horizon framework}
We use a rolling horizon framework to solve our MDP model. In general, the rolling horizon framework is a decision-making strategy in which a problem is repeatedly solved over a limited planning horizon that moves forward in time. At each step, a decision is implemented, and the horizon is shifted forward, allowing for continuous adaptation to new information and changing conditions. 
One approach to solve the operational route planning problem under uncertainty is to use a rolling horizon look-ahead framework in which the dynamic and stochastic nature of future requests is approximated as static but stochastic. 
Specifically, we decompose the problem of maximizing the total operators' profit at the end of the problem time horizon $T$ into multiple subsequent stochastic optimization problems at each decision time $\theta$ at which a new request $r_{\theta}$ arises. In these stochastic programs, we approximate the dynamic and stochastic future passenger requests through static and stochastic sample paths.
Another approach is to ignore all information about future passenger requests and use a myopic strategy. 

\subsection{Full information approximation}
\label{sec: full information}
In the following, we use a static approximation to capture the dynamic nature of our problem. To do so, at each decision time $\theta$, we assume a probability distribution that describes the likelihood of each potential set of passenger requests occurring. We then optimize our decision at $\theta$ given this distribution by solving a two-stage stochastic program. 

\subsubsection*{Two-stage stochastic program}
We define $y_{r_{\theta}}$ as the accept or reject decisions of the new passenger request $r_{\theta}$ at decision time~${\theta}$. Additionally, let $S_{\theta}$ be all passenger requests accepted up until decision time ${\theta}$, and let $r_{\theta}$ be the new passenger request as defined in Section~\ref{sec: Problem Setting}. Furthermore, we define $\bm{\xi_{\theta}}$ as the probability distribution of possible sample paths at decision time ${\theta}$ conditioned on the already accepted requests $S_\theta$ and $R(\bm{\xi_{\theta}})$ as a realization of $\bm{\xi_{\theta}}$. Note that each realization $R(\bm{\xi_{\theta}})$ is a sample path as defined in Section \ref{subsec: problem setting dynamic}, i.e., a set of requests that consists of all already accepted requests $S_{\theta}$, the new passenger request $r_{\theta}$, and possible future requests. Consequently, it holds that $r_{\theta} \in R(\bm{\xi_{\theta}})$ and $r \in R(\bm{\xi_{\theta}})$ for $r \in S_{\theta}$. Problems~\ref{prob: two stage} and~\ref{prob: two stage Q} define the first and second stages of our stochastic program for each decision time $\theta$ that we solve when using our full information approximation. 
Note that instead of including the utility of serving all previously accepted requests and potentially serving our current passenger request in the first stage of our stochastic programs, we choose to include this utility in the first sum of the objective function of the second stages $Q(S_{\theta}, y_{r_{\theta}}, \bm{\xi_{\theta}})$ (see Problem~\ref{prob: two stage Q}). Even though this utility is deterministic, i.e., independent from the distribution of future passenger requests, this notation allows for an easier decomposition into independent full information problems in Section~\ref{chapter: sample based decomposition}.

{
\allowdisplaybreaks
\begin{subequations}
\begin{align}
\max \quad \mathbb{E} [Q(S_{\theta}, y_{r_{\theta}}, \bm{\xi_{\theta}})] \label{obj: two-stage}
\end{align}
\begin{align}
y_{r_{\theta}} & \in \{0, 1\}
\end{align}
\label{prob: two stage}
\end{subequations}
}

{
\allowdisplaybreaks
\begin{subequations}
\begin{flalign}
Q(S_{\theta}, y_{r_{\theta}}, \bm{\xi_{\theta}}) = \max \quad \sum_{r \in \xi_{\theta}} u({r}) y_{r} -\sum_{i, j \in A} c_{ij} x_{ij} \label{obj: two-stage Q}&&
\end{flalign}
\begin{align}
y_{r} & = 1 & \forall r \in S_{\theta} \label{eq: two stage accept}\\
y_{r} - \sum_{i \in N} \sum_{j \in N^+(i)} \delta_{o(r)}^i x_{ij} & \leq 0 & \forall r \in R(\bm{\xi_{\theta}})\label{eq: two stage origin}\\
y_{r} - \sum_{j \in N} \sum_{i \in N^-(j)} \delta_{d(r)}^j x_{ij}& \leq 0  & \forall r \in R(\bm{\xi_{\theta}}) \label{eq: two stage destination}\\
\sum_{j \in N^+(i)} x_{ij} - \sum_{j \in N^-(i)} x_{ji} & = d_i  & \forall i \in N \label{eq: two stage flow conservation}\\
t_h + \sum_{i,j \in A_h} \tau_{ij} x_{ij} & \leq t_{h+1} &  h=1, \ldots, n \label{eq: two stage node times}\\
a_h & \leq t_h \leq b_h &  h=1, \ldots, n+1 \label{eq: two stage time windows}\\
\sum_{i \in Q} \sum_{j \neq i, j \in Q} x_{ij} & \leq \vert Q \vert - 1 & \forall Q \subseteq N_h, \vert Q \vert \geq 2, h = 1, \ldots, n \label{eq: two stage subtour elim}\\
y_{r} & \in \{0,1\} &  \forall r \in R(\bm{\xi_{\theta}}) \label{eq: two stage dimension y}\\
x_{ij} & \in \{0,1\} &  \forall i,j \in A\label{eq: two stage dimension x}
\end{align}
\label{prob: two stage Q}
\end{subequations}
}

Problem~\ref{prob: two stage} maximizes the expected profit based on the decision $y_{r_{\theta}}$ at decision time $\theta$, all previsously accepted passengers $S_{\theta}$, and the probability distribution of future passenger requests $\bm{\xi_{\theta}}$.
The objective function \eqref{obj: two-stage Q} aims to maximize the expected operator profit under the assumption that passenger request $r_{\theta}$ is either accepted $(y_{r_{\theta}} = 1)$ or rejected $(y_{r_{\theta}} = 0)$ in the first stage, considering previously accepted requests and future passenger requests at decision time $\theta$.
Constraints~\eqref{eq: two stage accept} ensure all requests accepted before decision time $\theta$ are served. Similarly to Problem~\eqref{prob: offline arcs}, Constraints~\eqref{eq: two stage origin} and \eqref{eq: two stage destination} ensure that a request $r \in R(\bm{\xi_{\theta}})$ can only be served if one stop out of the set of on-boarding stops $s(r)$ and one stop of the set of off-boarding stops $d(r)$ are visited by the vehicle. Constraints~\eqref{eq: two stage flow conservation} ensure flow conservation at all optional and compulsory stops. Constraints~\eqref{eq: two stage node times} define the starting times $t_h$ at all compulsory stops, and Constraints~\eqref{eq: two stage time windows} ensure that the time windows at the compulsory stops are met. Constraints~\eqref{eq: two stage subtour elim} ensure that no subtours within the segments exist and are ensured through lazy constraints. Finally, Constraints~\eqref{eq: two stage dimension y} and~\eqref{eq: two stage dimension x} define the domains of $y_r$ and $x_{ij}$.

\subsection*{Sample-based approximation of the two-stage stochastic program}
\label{chapter: sample-based approximation}
We utilize a sample-based approximation to tackle the computational complexity arising from computing expected second stage costs of the stochastic program. 
Let $\Omega$ define a set of possible request scenarios with the probability of occurrence of $p_{\sigma}$ for all $\sigma \in \Omega$. Here, for each $\sigma \in \Omega$ it holds that $S_{\theta} \subseteq \sigma$ and $r_{\theta} \in \sigma$. We can then approximate Problem~\ref{prob: two stage} through a multi-scenario deterministic model, in which we approximate the random variable $\bm{\xi_{\theta}}$ through a set of scenarios $\Omega$ with corresponding request sets $R^\sigma_{\theta}$ for each scenario $\sigma \in \Omega$.
Problem~\ref{prob: two stage Q scenarios} shows the sample-based approximation of Problem~\ref{prob: two stage Q}.

\vspace{-2ex}
{
\allowdisplaybreaks
\begin{subequations}
\begin{flalign}
Q^\Omega(S_{\theta}, y_{r_{\theta}}, \Omega) = \max \quad \sum_{\sigma\in \Omega} p_\sigma \Bigl(\sum_{r^\sigma \in R^\sigma_{\theta}} u({r^\sigma}) y_{r^\sigma} -\sum_{i, j \in A} c_{ij} x^\sigma_{ij} \Bigr) \label{obj: two-stage Q scenarios}&&
\end{flalign}
\vspace{-4ex}
\begin{align}
y_{r_{\theta}^\sigma} & = y_{r_{\theta}} & \label{eq: linking two stage Q scenarios}\\
y_{r^\sigma} & = 1 & \forall r^\sigma \in S_{\theta}, \forall \sigma\in \Omega \label{eq: two stage accept scenarios}\\
y_{r^\sigma} - \sum_{i \in N} \sum_{j \in N^+(i)} \delta_{o(r^\sigma)}^i x^\sigma_{ij} & \leq 0 & \forall r^\sigma \in R^\sigma_{\theta}, \forall \sigma\in \Omega\label{eq: two stage origin scenarios}\\
y_{r^\sigma} - \sum_{j \in N} \sum_{i \in N^-(j)} \delta_{d(r^\sigma)}^j x^\sigma_{ij}& \leq 0  & \forall r^\sigma \in R^\sigma_{\theta}, \forall \sigma\in \Omega \label{eq: two stage destination scenarios}\\
\sum_{j \in N^+(i)} x^\sigma_{ij} - \sum_{j \in N^-(i)} x^\sigma_{ji} & = d_i  & \forall i \in N, \forall \sigma\in \Omega \label{eq: two stage flow conservation scenarios}\\
t^\sigma_h + \sum_{i,j \in A_h} \tau_{ij} x^\sigma_{ij} & \leq t^\sigma_{h+1} &  h=1, \ldots, n, \forall \sigma\in \Omega \label{eq: two stage node times scenarios}\\
a_h & \leq t^\sigma_h \leq b_h &  h=1, \ldots, n+1 \label{eq: two stage time windows scenarios}\\
\sum_{i \in Q} \sum_{j \neq i, j \in Q} x^\sigma_{ij} & \leq \vert Q \vert - 1 & \forall Q \subseteq N_h, \vert Q \vert \geq 2, h = 1, \ldots, n, \forall \sigma\in \Omega \label{eq: two stage subtour elim}\\
y_{r^\sigma} & \in \{0,1\} &  \forall r^\sigma \in R^\sigma_{\theta}, \forall \sigma\in \Omega \label{eq: two stage dimension y scenarios}\\
x^\sigma_{ij} & \in \{0,1\} &  \forall i,j \in A, \forall \sigma\in \Omega\label{eq: two stage dimension x scenarios}\\
y_{r_{\theta}} & \in \{0, 1\}
\end{align}
\label{prob: two stage Q scenarios}
\end{subequations}
}

The objective function~\eqref{obj: two-stage Q scenarios} calculates the expected operators' profit over all scenarios $\sigma \in \Omega$. Additionally, Constraint~\eqref{eq: linking two stage Q scenarios} ensures, that over all scenarios, request $r_{\theta}$ is either always accepted ($y_{r_{\theta}} = 1$) or never accepted ($y_{r_{\theta}} = 0$). Here, we introduce variables $y_{r_{\theta}^\sigma}$ as this notation enables us to describe our heuristic approach in Section~\ref{chapter: heuristic}.

Figure~\ref{fig: sample based approximation} visualizes the rational of this sample-based approximation approach in which we consider all possible future request scenarios at the same time.

\begin{figure}[t]
\centering
    \includegraphics[width=0.7\columnwidth]{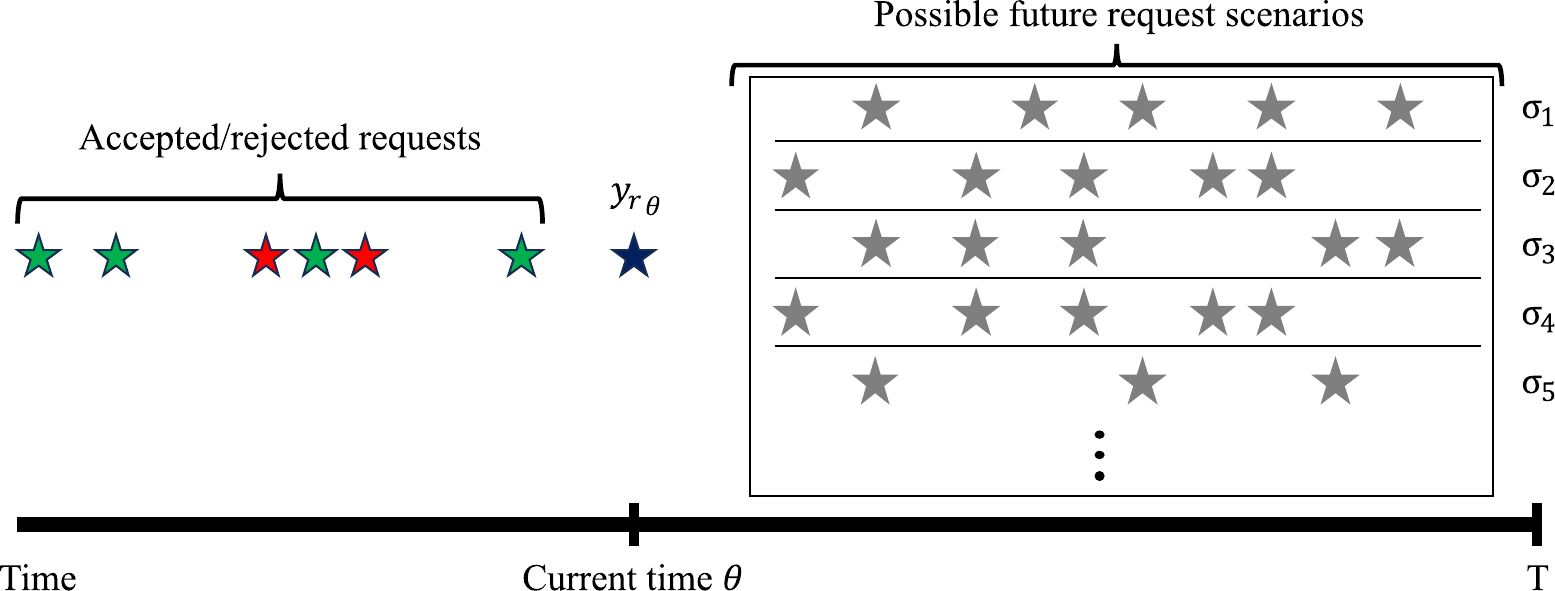}
\caption{Visualization of our sample-based approximation where we consider possible future request scenarios at the same time}
\label{fig: sample based approximation}
\end{figure}

\subsubsection{Decomposition of sample-based approximation}
\label{chapter: sample based decomposition}
Solving the sample-based approximation (Problem~\ref{prob: two stage Q scenarios}) is time-consuming. As we aim to take a decision fast and immediately inform the passenger whether we accept or reject their request, we show how to decompose the sample-based approximation (Problem~\ref{prob: two stage Q scenarios}) into multiple independent full information planning problems (Problems~\ref{prob: offline arcs}) in the following. Note that even though the following proof of the decomposition is rather intuitive, it is necessary to substantiate our scenario-based decomposition and consensus-based heuristic. In the following, we branch on variable $y_{r_{\theta}}$ and thus, implicitly, on the second stage variables $y_{r_{\theta}^\sigma}$, $\forall \sigma \in \Omega$. Accordingly, Constraints~\eqref{eq: linking two stage Q scenarios} automatically hold $\forall \sigma \in \Omega$, and we can decompose $Q^\Omega(S_{\theta}, y_{r_{\theta}}, \Omega)$ by scenario as follows.

{
\allowdisplaybreaks

\begin{align}
Q^\Omega(S_{\theta}, y_{r_{\theta}}, \Omega) = \max \{\sum_{\sigma \in \Omega} p_\sigma Q^\sigma(S_{\theta}, y_{r_{\theta}}, \sigma):  y_{r_{\theta}} \in \{0, 1\}\}
\label{prob: decomp Q}
\end{align}

}

Here, we define $Q^\sigma(S_{\theta}, y_{r_{\theta}}, \sigma)$ as the optimal objective value of Problem~\ref{prob: two stage Q decomp} with $y_{r_{\theta}} = y_{r_{\theta}^\sigma}$ if Problem~\ref{prob: two stage Q decomp} is feasible and $-\infty$ otherwise. Through Theorem~\ref{theorem: infeasible subproblem}, we show that Problem~\ref{prob: two stage Q decomp} being feasible is independent of scenario $\sigma$. As we assume that serving no requests is always feasible, we ensure that $Q^\Omega(S_{\theta}, y_{r_{\theta}}, \Omega) > -\infty$ for all $\theta < T$. Accordingly, we always ensure a feasible solution at each decision time $\theta$. 

\begin{theorem}
\label{theorem: infeasible subproblem}
    Problem~\ref{prob: two stage Q decomp} is infeasible for $\sigma \in \Omega$ with $y_{r_{\theta}} = 1$ if and only if Problem~\ref{prob: two stage Q decomp} is infeasible for $S_{\theta} \cup r_{\theta}$ with $y_{r_{\theta}} = 1$.
\end{theorem}

\begin{proof}
    Let $\sigma = S_{\theta} \cup r_{\theta} \cup \mathcal{F}$ with $\mathcal{F}$ being a sample of possible future requests. We define $\Lambda(\sigma)$ as the feasible region of Problem~\ref{prob: two stage Q decomp} for $\sigma \in \Omega$ and $\Lambda(S_{\theta} \cup r_{\theta})$ as the feasible region for $S_{\theta} \cup r_{\theta}$ with $y_{r_{\theta}} = 1$. For the Theorem to hold, we have to show that $\Lambda(\sigma) = \emptyset$ if and only if $\Lambda(S_{\theta} \cup r_{\theta}) = \emptyset$. We assume that serving no customers is always feasible. Accordingly, $\Lambda(S_{\theta} \cup r_{\theta})$ is a facet of $\Lambda(\sigma)$ with $\Lambda(\sigma) \cap \{y_r = 0: r \in \mathcal{F}\} = \Lambda(S_{\theta} \cup r_{\theta})$. It follows directly that $\Lambda(\sigma) = \emptyset \Leftrightarrow \Lambda(S_{\theta} \cup r_{\theta}) = \emptyset$. \qed
\end{proof}

\smallskip Additionally, we define Problem~\ref{prob: two stage Q decomp} as follows.

{
\allowdisplaybreaks
\begin{subequations}
\begin{flalign}
Q^\sigma(S_{\theta}, y_{r_{\theta}}, \sigma) = \max \quad \sum_{r^\sigma \in R^\sigma_{\theta}} u({r^\sigma}) y_{r^\sigma} -\sum_{i, j \in A} c_{ij} x^\sigma_{ij} \label{obj: two-stage Q decomp}&&
\end{flalign}
\begin{align}
y_{r_{\theta}^\sigma} & = y_{r_{\theta}} & \label{eq: linking two stage Q decomp}\\
y_{r^\sigma} & = 1 & \forall r^\sigma \in S_{\theta} \label{eq: two stage accept decomp}\\
y_{r^\sigma} - \sum_{i \in N} \sum_{j \in N^+(i)} \delta_{o(r^\sigma)}^i x^\sigma_{ij} & \leq 0 & \forall r^\sigma \in R^\sigma_{\theta}\label{eq: two stage origin decomp}\\
y_{r^\sigma} - \sum_{j \in N} \sum_{i \in N^-(j)} \delta_{d(r^\sigma)}^j x^\sigma_{ij}& \leq 0  & \forall r^\sigma \in R^\sigma_{\theta} \label{eq: two stage destination decomp}\\
\sum_{j \in N^+(i)} x^\sigma_{ij} - \sum_{j \in N^-(i)} x^\sigma_{ji} & = d_i  & \forall i \in N \label{eq: two stage flow conservation decomp}\\
t^\sigma_h + \sum_{i,j \in A_h} \tau_{ij} x^\sigma_{ij} & \leq t^\sigma_{h+1} &  h=1, \ldots, n \label{eq: two stage node times decomp}\\
a_h & \leq t^\sigma_h \leq b_h &  h=1, \ldots, n+1 \label{eq: two stage time windows decomp}\\
y_{r^\sigma} & \in \{0,1\} &  \forall r^\sigma \in R^\sigma_{\theta} \label{eq: two stage dimension y decomp}\\
x^\sigma_{ij} & \in \{0,1\} &  \forall i,j \in A\label{eq: two stage dimension x decomp}
\end{align}
\label{prob: two stage Q decomp}
\end{subequations}
}

Note that in Problem~\ref{prob: two stage Q decomp} $y_{r_{\theta}}$ and $\sigma$ are fixed. Consequently, the objective function~\eqref{obj: two-stage Q decomp} calculates the operators' profit for the request set $R^\sigma_{\theta}$ of scenario $\sigma$ under the assumption that decision $y_{r_{\theta}}$ has been taken \eqref{eq: linking two stage Q decomp} and all requests accepted before decision time $\theta$ are served \eqref{eq: two stage accept decomp}. 
Furthermore, Constraints~\eqref{eq: two stage origin decomp} and \eqref{eq: two stage destination decomp} ensure that a request $r^\sigma \in R^\sigma_{\theta}$ can only be served if one stop out of the set of on-boarding stops $s(r)$
and one stop of the set of off-boarding stops $d(r)$ are visited by the vehicle. Constraints~\eqref{eq: two stage flow conservation decomp} ensure
flow conservation at all optional and compulsory stops. 

\begin{figure}[!t]
\centering
    \includegraphics[width=0.7\columnwidth]{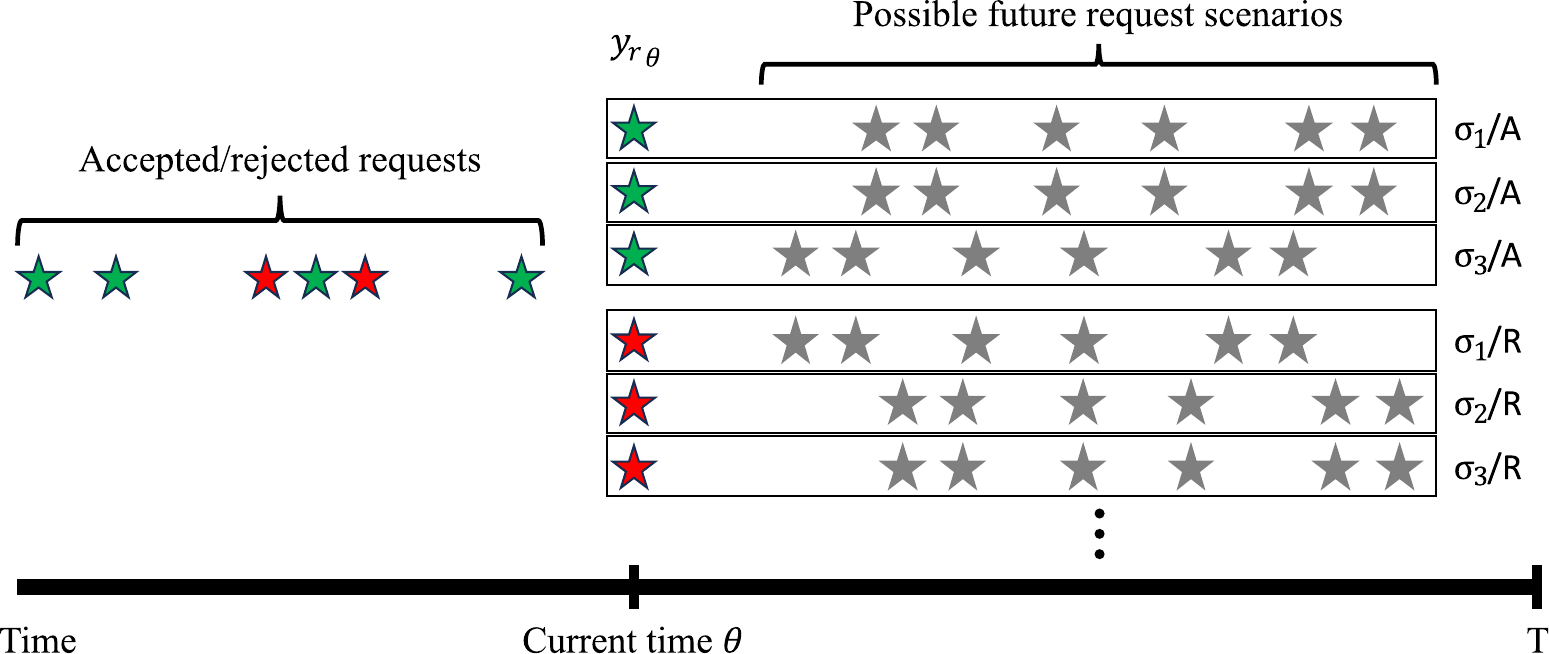}
\caption{Visualization of our sample-based approximation decomposition into different subproblems where we independently solve different scenarios $\sigma_i$ with acceptance (A) or rejection (R) of request $r_{\theta}$}
\label{fig: decomposition exact}
\end{figure}

Constraints~\eqref{eq: two stage node times decomp} define the starting times at all compulsory stops, and Constraints~\eqref{eq: two stage time windows decomp} ensure that the time windows at the compulsory
stops are met. Finally, Constraints~\eqref{eq: two stage dimension y decomp} and \eqref{eq: two stage dimension x decomp} define the dimensions of $y_{r^\sigma}$ and $x_{ij}^\sigma$.
Accordingly, Problem~\ref{prob: two stage Q decomp} is now equivalent to our full information Problem~\ref{prob: offline arcs} with all previously accepted requests' variables $r^\sigma \in S_{\theta}$ fixed to one in Constraints~\eqref{eq: two stage accept decomp} and our current request variable $y_{r_{\theta}}$ either fixed to zero or one. 

Figure~\ref{fig: decomposition exact} visualizes how our decomposition allows us to solve all scenarios $\sigma \in \Omega$ independently in a full information problem setting, once assuming we reject the current passenger request ($y^\sigma_{r_{\theta}} = 0$) and once assuming we accept it ($y^\sigma_{r_{\theta}} = 1$). Afterward, we either accept or reject the current passenger request depending on whether $\sum_{\sigma \in \Omega} p_\sigma Q^\sigma(S_{\theta}, 1, \sigma) > \sum_{\sigma \in \Omega} p_{\sigma} Q^\sigma(S_{\theta}, 0, \sigma)$ (accept) or not (reject).

Note that solving all full information subproblems independently is easily parallelizable. Indeed, we will show in Section~\ref{sec: Results} that this decomposition allows us to make a decision much faster than solving Problem~\ref{prob: two stage Q scenarios} directly. 

\subsubsection{Consensus-based heuristic approach}
\label{chapter: heuristic}
Our decomposition of the operational route planning problem under uncertainty allows us to solve multiple full information problems, once for accepting the current passenger request and once for rejecting it. Alternatively, Figure~\ref{fig: decomposition heuristic} visualizes how we can also solve the individual full information problems without fixing the current passenger requests variable $y_{r_{\theta}^\sigma}$, i.e., neglecting Constraint~\eqref{eq: linking two stage Q decomp}, and take a weighted majority vote to decide on whether to accept or reject the current passenger request.

Such a weighted majority vote counteracts instances in which there exist scenarios that achieve disproportionately large profits. Our conjecture is that in our original decomposition, these scenarios may influence the acceptance decision substantially, even if the probability of occurrence is low. Deciding whether to accept or reject passenger request $r_{\theta}$ based on a weighted majority vote decouples the decision from the absolute objective value.

{
\allowdisplaybreaks
\vspace{-4ex}
\begin{align}
y_{r_{\theta}} =
\begin{cases}
    1, & \text{if } \sum_{\sigma \in \Omega} p_\sigma \mathbbm{1}_{\{y_{r_{\theta}^\sigma} = 1\}} \geq \frac{\vert \Omega \vert}{2} , \\
    0, & \text{otherwise }
\end{cases}
\end{align}
\label{def: heuristic decision}

}

Here, for $\sigma \in \Omega$ we define the indicator function $\mathbbm{1}_{\{y_{r_{\theta}^\sigma} = 1\}}$ as follows

{
\allowdisplaybreaks
\vspace{-4ex}
\begin{align}
\mathbbm{1}_{\{y_{r_{\theta}^\sigma} = 1\}} =
\begin{cases}
    1, & \text{if } y_{r_{\theta}^\sigma} = 1 , \\
    0, & \text{if } y_{r_{\theta}^\sigma} = 0
\end{cases}
\end{align}
\label{def: indicator function}
}

Similar to our previous decomposition approach, solving the independent subproblems of our heuristic approach is easily parallelizable. 

\begin{figure}[!t]
\centering
    \includegraphics[width=0.8\columnwidth]{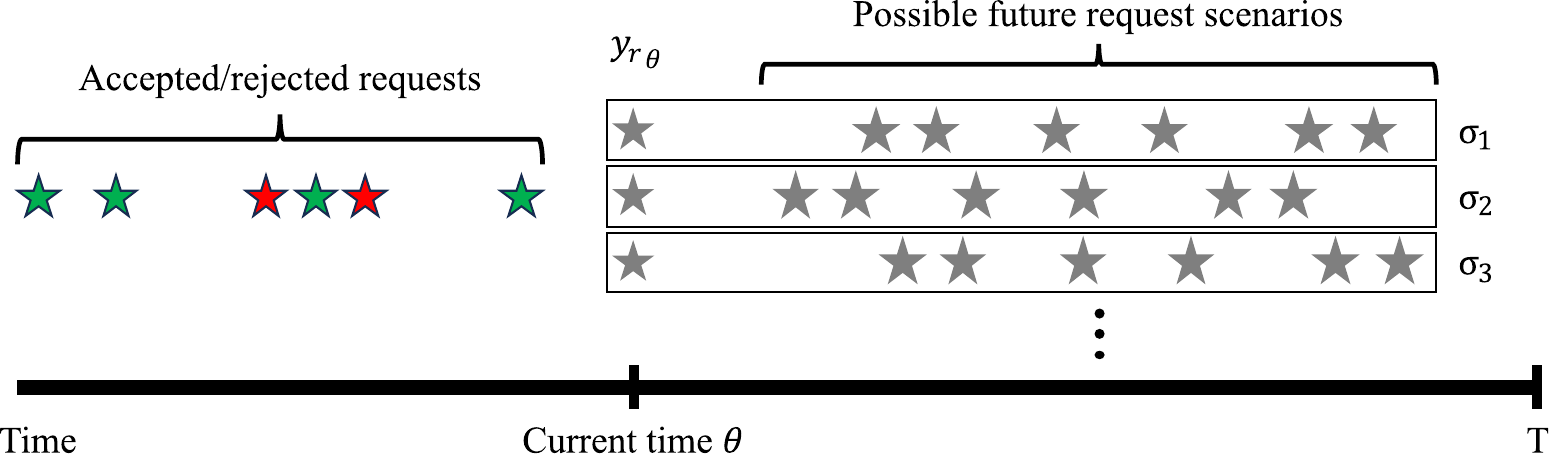}
\caption{Visualization of our heuristic approach where we independently solve different scenarios $\sigma_i$ without fixing the decision for $r_{\theta}$}
\label{fig: decomposition heuristic}
\end{figure}

\subsection{Myopic approach}
\label{chapter: greedy}
So far, we always considered future passenger requests and approximated them through a stochastic distribution. 
Instead, we can also make a decision whether to accept or reject passenger request $r_{\theta}$ based on a myopic approach. Figure~\ref{fig: greedy approach} visualizes how in such an approach, we neglect all information about future passenger requests and accept passenger requests $r_{\theta}$ only if this immediately increases our profit. 
To do so, we solve Problem~\ref{prob: two stage Q decomp} only for one scenario $\sigma = S_{\theta} \cup r_{\theta}$. Accordingly, we decide whether to accept or reject passenger request $r_{\theta}$ based on Equation~\eqref{def: greedy heuristic decision} in accordance to Constraint~\eqref{eq: linking two stage Q decomp}.

{
\allowdisplaybreaks
\vspace{-4ex}
\begin{align}
y_{r_{\theta}} =
\begin{cases}
    1, & \text{if } y_{r_{\theta}^\sigma} = 1 , \\
    0, & \text{otherwise }
\end{cases}
\label{def: greedy heuristic decision}
\end{align}

}

\begin{figure}[h]
\centering
    \includegraphics[width=0.7\columnwidth]{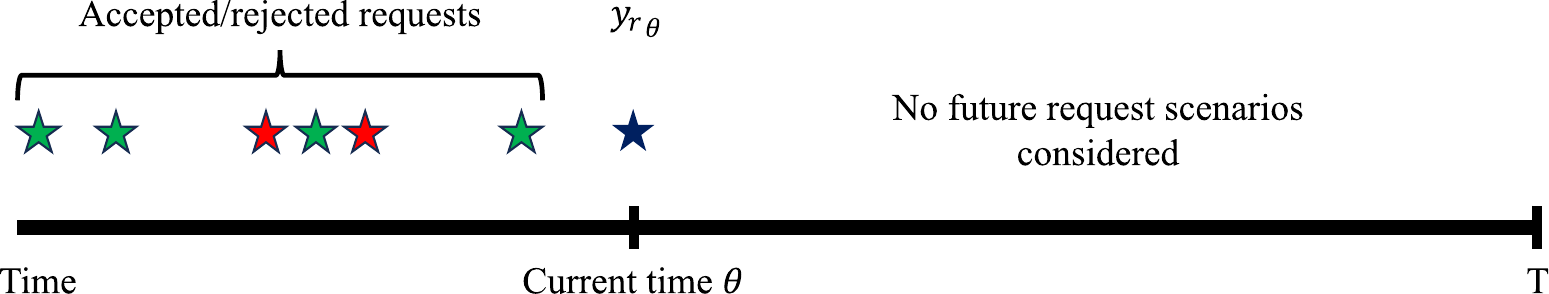}
\caption{Visualization of our myopic approach where we do not consider future request scenarios}
\label{fig: greedy approach}
\end{figure}

\section{Computational analysis}
\label{sec: Results}

Our analysis is fivefold. First, we analyze the computational efficiency of our \gls{2S-SP}, our \gls{HA}, and the Myopic approach. To do so, we first compare the computation times, profits, and percentages of passengers served in solutions produced by those approaches on two different \gls{DAS} lines with different degrees of route flexibility. 
Second, we perform a sensitivity analysis on the number of scenarios in our \gls{2S-SP} and \gls{HA}. Here, we analyze the change in computation times, profits, and served passengers attained through our \gls{2S-SP} and \gls{HA} approaches under different numbers of future demand scenarios. 
Third, we evaluate the solution quality of our approaches compared to an optimal solution under full information. To do so, we compare our \gls{2S-SP}, \gls{HA}, and Myopic approaches to the profits obtained and passengers served under full information, solved with our full information model (Problem~\ref{prob: offline arcs}). 
Fourth, we analyze the benefits and drawbacks of converting a fixed-line bus system to a \gls{DAS}. To do so, we compare both systems' profits, served passengers, and travel distances to show how an increase in travel distances translates to an increase in profits and served passengers. 
Fifth, we aim to understand how the passenger's willingness to walk to nearby stops and the flexibility of the bus route influences the profit and number of served passengers. Here, we show the differences in profits and number of served passengers under varying passenger walking willingness and route flexibility. 

\subsection*{Instances}
We use data provided by Stadtwerke München GmbH (SWM), the public transportation provider of Munich. This SWM data consists of the average number of passengers boarding and aligning at the different stops of two different fixed bus routes over the entire day. The data was collected between April and June 2022. Based on this data, we generate between five and 40 demand scenarios in three steps. First, we randomly determine an OD-matrix between the stops of the original fixed route through iterative proportional fitting \citep{ruschendorf1995convergence} based on the SWM data. This OD-matrix indicates for each pair of compulsory stops the number of passengers seeking transportation between those stops. Here, the origins and destinations of all passengers lie at the stops of the original fixed bus route. Second, based on the determined OD-matrix, we randomly map passenger origins and destinations onto buildings within 300 meters of their respective original origin and destination. These buildings correspond to the pick-up and drop-off locations of a passenger request and, in combination with the walking distance, later define the pick-up and drop-off sets of the passenger requests. Note that if there are multiple requests between the same origin and destination in the OD-matrix, we individually map each request to their own pick-up and drop-off locations. Third, to be able to model a point process in which the requests arrive in the system over time, we use random sampling to obtain a request time between zero and 10,800 seconds (3 hours) before the route's departure time for each passenger request.

Besides the generation of passenger demand, we also utilize the SWM data to determine the compulsory and optional stops of our route. To do so, we utilize a \textit{compulsory stop factor} (CSF), which determines the degree of flexibility of our \gls{DAS} line. This CSF lies between zero and one. The first and last stops of the original fixed bus route are always marked as compulsory stops as they often lie at critical transfer hubs or are essential for operations, e.g., driver breaks. Additionally, we order the stops of the original fixed route based on their passenger demand utilization and mark a percentage of them as compulsory. This percentage corresponds to the CSF. 
Furthermore, we allow each intersection within 600 meters of the original fixed bus route to be an optional stop. Here, we merge optional stops within 200 meters of each other. 

\subsection*{Computational setting}
All our experiments have been conducted on a standard desktop computer equipped with two \texttt{Intel(R) Core(TM) i9-9900, 3.1 GHz CPU} and a total of \texttt{4 GB} of \texttt{RAM}, running \texttt{Ubuntu 20.04}. We have implemented all algorithms in \texttt{Python (3.8.11)} using \texttt{Gurobi 10.0.0}. Our source code can be found on \url{https://github.com/tumBAIS/Route_Planning_DAS}.

\subsection{Computational effectiveness of the proposed methods}
The profit of our operational route planning problem under uncertainty is driven by two quantities. The revenue of accepted passenger requests and travel costs. To see what drives solution quality and understand the computational effectiveness of our proposed methods, we compare our proposed methods on two different routes based on the achieved profit, served passengers (SP) out of all issued requests, computation times in serial computation (ser.), and a theoretical lower bound of parallel computation times (par.). Here, the theoretical lower bounds of parallel computation times for \gls{2S-SP} and \gls{HA} are given by the maximum runtime of the independent subproblems for each request decision. We include this theoretical lower bound as parallel computation is a typical approach for dynamic decision-making in stochastic problem settings. We compare the different parameters by aggregating the results of solving five different instances for each instance configuration. To do so, we compare our proposed methods with five scenarios in our \gls{2S-SP} and \gls{HA} approaches, a uniform passenger utility of 750, and a CSF $\in [0.2, 0.4, 0.6, 0.8, 1]$. 
We analyze our results for different CSF and different numbers of scenarios in our \gls{2S-SP} and \gls{HA} approaches. Note that in our \gls{2S-SP}, we solve our decomposition Problem~\ref{prob: decomp Q}. We were not able to solve Problem~\ref{prob: two stage Q scenarios} within four hours for all decision times for our configurations. 

\subsubsection{Computational effectiveness for varying CSF}
In this section, we focus on how the degree of flexibility of a \gls{DAS} line, i.e., the CSF, changes the computational effectiveness of our three approaches.
Table~\ref{tab: comparison computational effectiveness 20 scenarios route 55} shows that for Route 55, our serial-computed \gls{HA} approach is between 22 - 54\% faster than our \gls{2S-SP} approach, and our Myopic approach is between 88 - 98\% faster than our \gls{HA}. Comparing the theoretical lower bound of parallel computation, we see that our \gls{2S-SP} and \gls{HA} approaches perform similarly. Note that our \gls{2S-SP} still needs to solve twice the number of subproblems as our \gls{HA}. We also see that with increasing flexibility of a \gls{DAS} line, i.e., decreasing CSF, the parallel computation times of our \gls{2S-SP} and \gls{HA} approaches increase significantly. 

\begin{table}[!t]
  \centering
  \caption{Comparison between \gls{2S-SP}, \gls{HA}, and Myopic with different compulsory stop factors for Route 55 with 5 future demand scenarios}
  \label{tab: comparison computational effectiveness 20 scenarios route 55}
  \small
\begin{tabular}{cccccccccccc}
\multicolumn{1}{l}{} &  \multicolumn{4}{c}{\gls{2S-SP}} &  \multicolumn{4}{c}{\gls{HA}} &  \multicolumn{3}{c}{Myopic} \\
 \cmidrule(lr){2-5} \cmidrule(lr){6-9} \cmidrule(lr){10-12}
  CSF & Profit & SP [\%] & ser. [s] & par. [s] & Profit & SP & ser. & par. & Profit & SP & ser.\\
\midrule
0.20 & 25864.38 & 83.02 & 4.97 & 0.98 & 25792.82 & 82.87 & 2.28 & 0.89 & 25316.64 & 82.10 & 0.25 \\
0.40 & 23456.57 & 77.31 & 2.07 & 0.36 & 23290.62 & 76.85 & 1.41 & 0.42 & 22166.70 & 75.00 & 0.14 \\
0.60 & 22960.73 & 75.46 & 1.13 & 0.20 & 22956.19 & 75.46 & 0.88 & 0.24 & 22429.49 & 75.22 & 0.09 \\
0.80 & 21472.71 & 72.53 & 0.25 & 0.03 & 21366.03 & 72.07 & 0.17 & 0.04 & 18658.58 & 63.85 & 0.02 \\
1.00 & 20656.19 & 69.91 & 0.10 & 0.01 & 20561.83 & 69.60 & 0.07 & 0.01 & 20731.19 & 70.06 & 0.01 \\
\bottomrule
\end{tabular}
\end{table}

\begin{table}[!t]
  \centering
  \caption{Comparison between \gls{2S-SP}, \gls{HA}, and Myopic with different compulsory stop factors Route 155 with 5 future demand scenarios}
  \label{tab: comparison computational effectiveness 20 scenarios route 155}
  \small
\begin{tabular}{cccccccccccc}
\multicolumn{1}{l}{} &  \multicolumn{4}{c}{\gls{2S-SP}} &  \multicolumn{4}{c}{\gls{HA}} &  \multicolumn{3}{c}{Myopic} \\
 \cmidrule(lr){2-5} \cmidrule(lr){6-9} \cmidrule(lr){10-12}
 CSF & Profit & SP [\%] & ser. [s] & par. [s] & Profit & SP & ser. & par. & Profit & SP & ser.\\
\midrule
0.20 & 19181.68 & 89.92 & 3.99 & 0.61 & 19011.13 & 89.42 & 1.72 & 0.47 & 18214.41 & 87.43 & 0.16 \\
0.40 & 18575.22 & 87.41 & 1.44 & 0.23 & 18494.20 & 86.90 & 1.19 & 0.31 & 17806.07 & 84.13 & 0.11 \\
0.60 & 16878.16 & 81.11 & 0.56 & 0.06 & 16792.81 & 80.60 & 0.32 & 0.09 & 16807.90 & 80.86 & 0.05 \\
0.80 & 16498.82 & 79.09 & 0.12 & 0.01 & 16498.82 & 79.09 & 0.06 & 0.01 & 16360.81 & 78.84 & 0.01 \\
1.00 & 15282.68 & 74.81 & 0.05 & 0.01 & 15209.15 & 74.56 & 0.03 & 0.01 & 15209.15 & 74.56 & 0.00 \\
\bottomrule
\end{tabular}
\end{table}

Looking at our approaches' profits, we see that our \gls{2S-SP} achieves the best profit for all CSFs except 1.0, followed by our \gls{HA} approach. As expected, our Myopic approach performs worse than our \gls{2S-SP} and \gls{HA} approaches but surprisingly outperforms our \gls{2S-SP} and \gls{HA} at a CSF of 1.0. The reason for that is a higher problem complexity for instances with high flexibility where our \gls{2S-SP} and \gls{HA} approaches benefit from utilizing historical data. In low flexibility instances, the \gls{DAS} route is mostly fixed due to the time windows at the compulsory stops. We also see that our \gls{2S-SP} approach serves more passengers than our \gls{HA} and Myopic approaches.

Table~\ref{tab: comparison computational effectiveness 20 scenarios route 155} shows that for Route 155, in serial computation, our \gls{HA} approach is between 17 - 57\% faster than our \gls{2S-SP} approach, and our Myopic approach is between 83 - 91\% faster than our \gls{HA}. Again, the theoretical lower bounds of parallel computation of our \gls{2S-SP} and \gls{HA} approaches perform similarly and again increase with the growing flexibility of the \gls{DAS} line. Looking at our approaches' profits, we see that similar to Route 55, \gls{2S-SP} achieves the best profit for all CSFs. This can analogously be explained by the higher number of passengers served. 

\begin{observation}
    Regarding solution quality and the number of served passengers, our \gls{2S-SP} approach yields the best results, followed by our \gls{HA} and finally our Myopic approach. Additionally, our Myopic approach is the fastest in serial computation, followed by our \gls{HA} approach, which is between 17 - 57\% faster than our \gls{2S-SP} approach. Both our \gls{2S-SP} and \gls{HA} approaches have similar theoretical lower bounds of parallel computation.
\end{observation}

\begin{observation}
    For each algorithmic approach and instance the serial and parallel computation times of our three approaches increase significantly with growing route flexibility, i.e., a decrease of CSF. 
\end{observation}

\subsubsection{Distribution of computation times for varying CSF}

Figure~\ref{fig: computation time CSF} shows the distribution of serial and parallel computation times for our approaches for instances of Route 55. Note that the parallel computation times of our Myopic approach are equivalent to the serial computation times. 

Building on Table~\ref{tab: comparison computational effectiveness 20 scenarios route 55}, we see that our Myopic approach has stable computation times of under two seconds for all instances and under one second for instances with a CSF greater than 0.2. Our \gls{2S-SP} and \gls{HA} approaches are not as stable in serial and parallel computation times and range between $0 - 28$~seconds for serial computation and $0 - 5$~seconds for parallel computation. The reason for that is, first, the different numbers of scenarios in our \gls{2S-SP} and \gls{HA} approaches and, second, significant variance in computation times for the different subproblems of the requests. 

\begin{figure}[!b]
\centering
\vspace{-2ex}
	\begin{subfigure}[t]{.49\textwidth}
		\centering		
		\includegraphics[width=\textwidth]{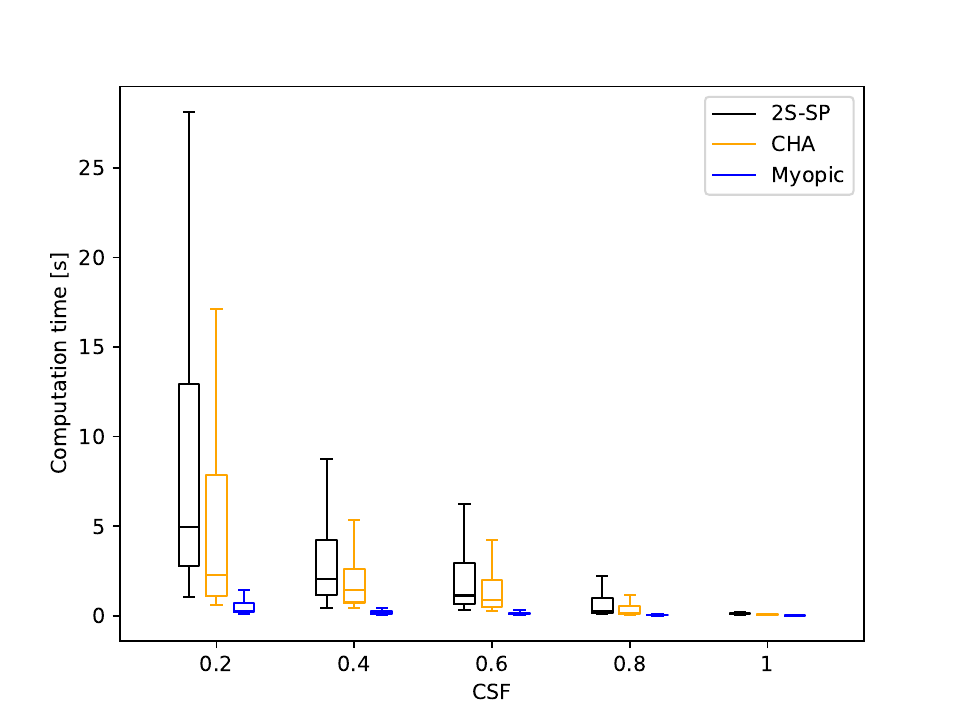}
        \captionsetup{format=hang}
		\caption{Serial computation times}
	\end{subfigure}
	\begin{subfigure}[t]{.49\textwidth}
		\centering		
		\includegraphics[width=\textwidth]{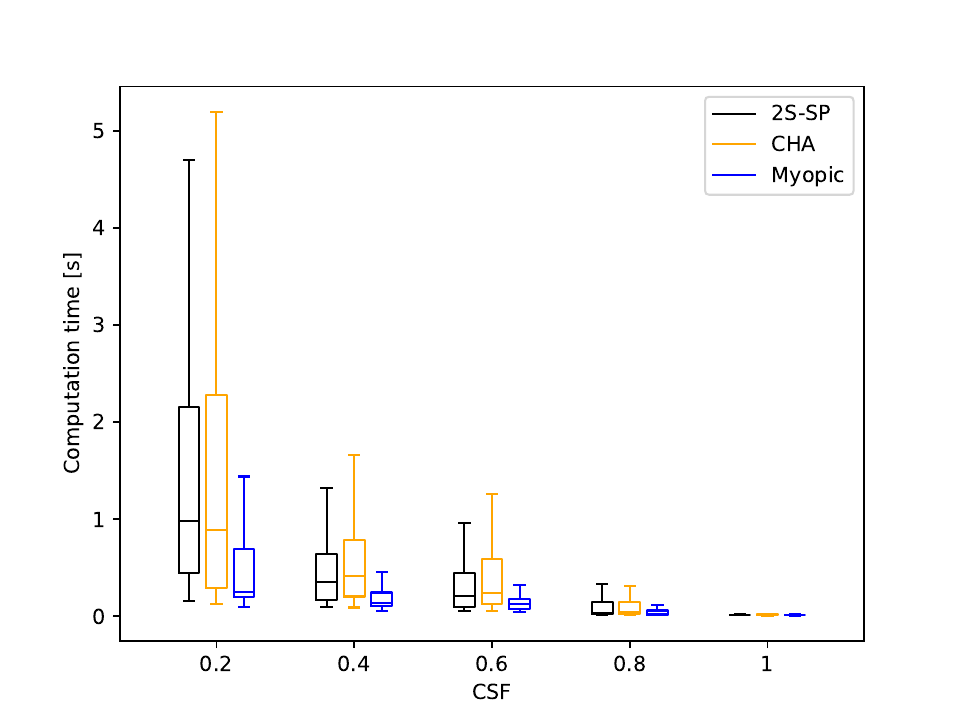}
        \captionsetup{format=hang}
		\caption{Parallel computation times}
	\end{subfigure}
 	\caption{Difference in serial and parallel computation times based on varying CSF for Route 55}
	\label{fig: computation time CSF}
\end{figure}

\begin{figure}[!b]
\centering
\vspace{-2ex}
	\begin{subfigure}[t]{.49\textwidth}
		\centering		
		\includegraphics[width=\textwidth]{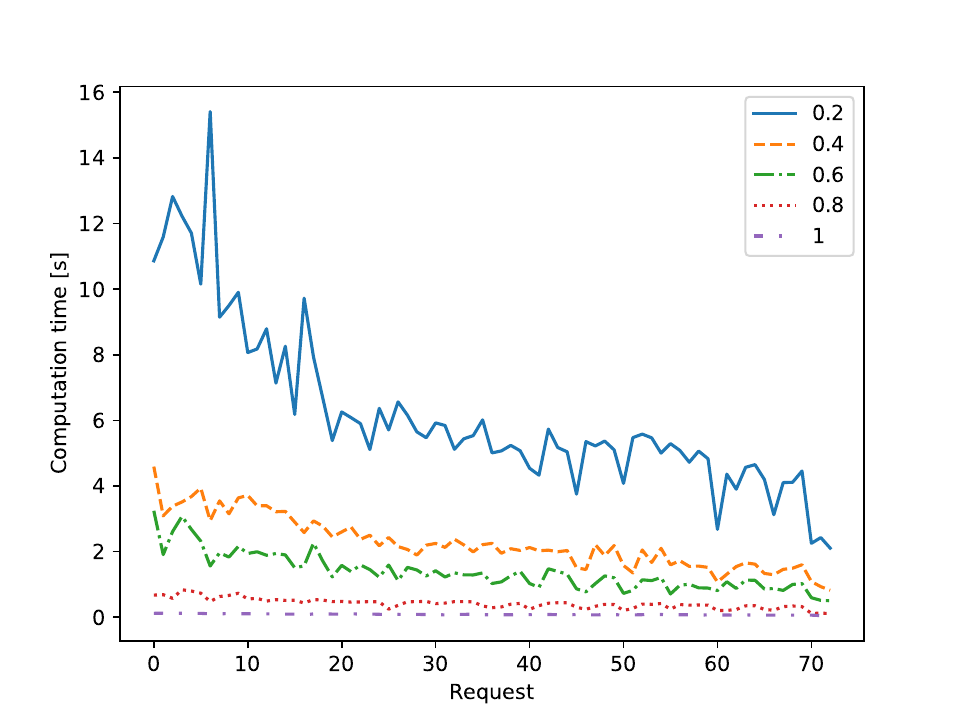}
        \captionsetup{format=hang}
		\caption{\centering Serial computation times for different \newline CSF (\gls{2S-SP})}
	\end{subfigure}
	\begin{subfigure}[t]{.49\textwidth}
		\centering		
		\includegraphics[width=\textwidth]{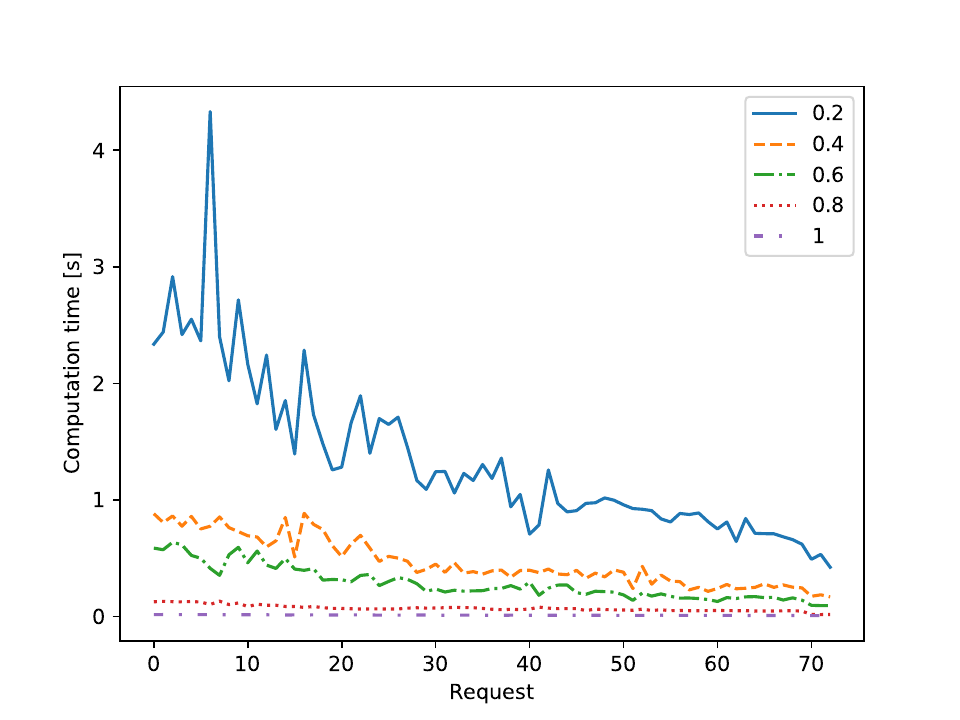}
        \captionsetup{format=hang}
		\caption{\centering Parallel computation times for different \newline CSF (\gls{2S-SP})}
	\end{subfigure}
    	\caption{Development of average serial and parallel computation times based on different CSF for Route 55}
    \label{fig: development computation time CSF}
\end{figure}

\begin{figure}[!t]\ContinuedFloat
 	\begin{subfigure}[t]{.49\textwidth}
		\centering		
		\includegraphics[width=\textwidth]{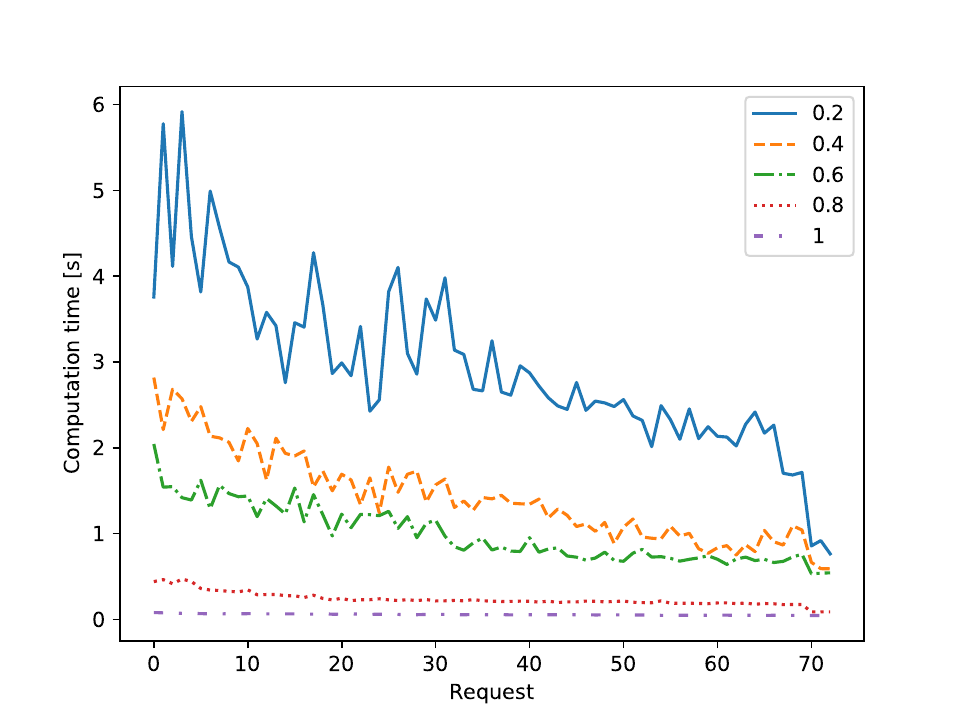}
        \captionsetup{format=hang}
		\caption{\centering Serial computation times for different \newline  CSF (\gls{HA})}
	\end{subfigure}
	\begin{subfigure}[t]{.49\textwidth}
		\centering		
		\includegraphics[width=\textwidth]{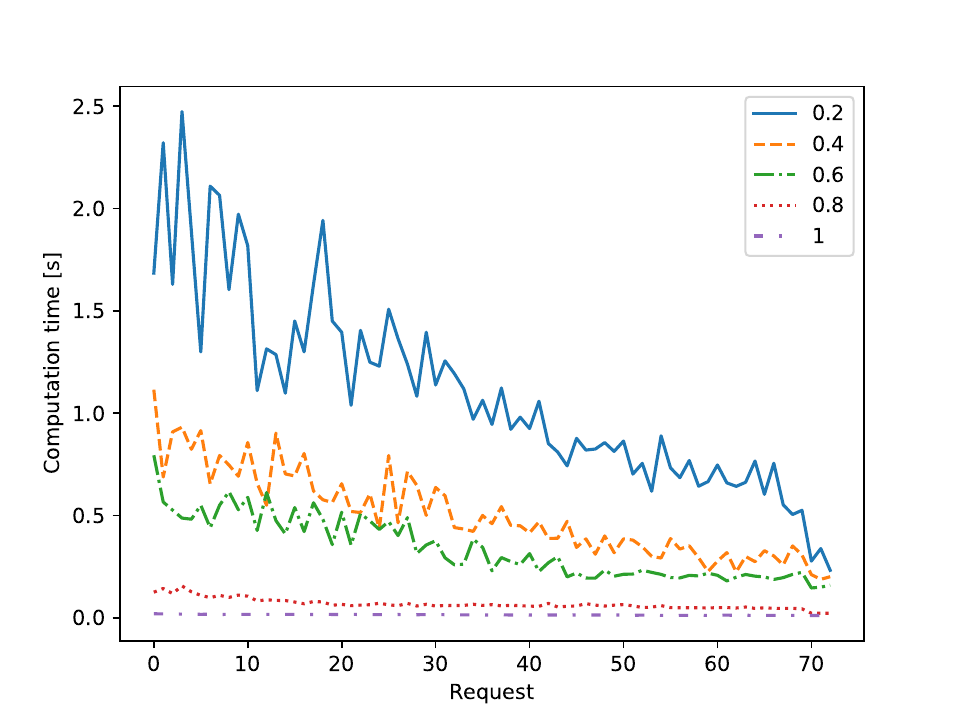}
        \captionsetup{format=hang}
		\caption{\centering Parallel computation times for different \newline CSF (\gls{HA})}
	\end{subfigure}

  \begin{subfigure}[t]{.49\textwidth}
		\centering		
		\includegraphics[width=\textwidth]{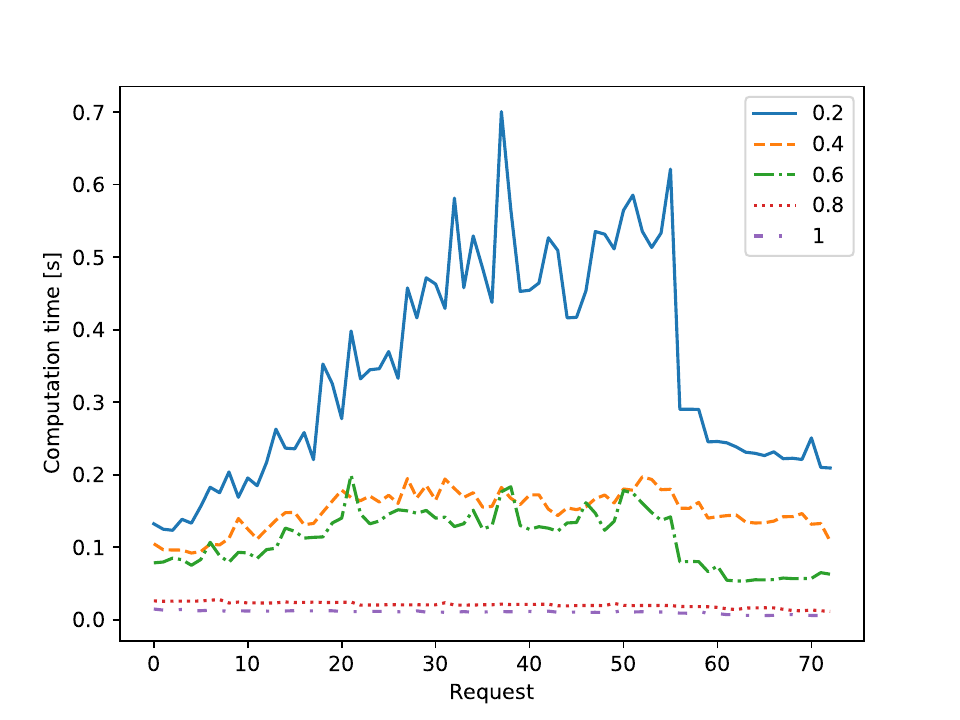}
        \captionsetup{format=hang}
		\caption{\centering Serial computation times for different \newline CSF (Myopic)}
	\end{subfigure}	
\end{figure}

Figure~\ref{fig: development computation time CSF} shows the difference in computation times for the different subproblems of the requests. 
We see that the more decisions on requests have already been taken, the faster our approaches can solve the subproblem of the current passenger request. 
Additionally, the later we have to decide on accepting or rejecting a request, the fewer future requests we have to consider in the resulting subproblems. The reduction in computation times over the course of our operational planning problem under uncertainty holds for both serial and parallel computation times for our \gls{2S-SP} and \gls{HA} approaches. 

Interestingly, our Myopic approach shows exactly the opposite development of computation times, i.e., an increase in computation times over the course of our operational planning problem under uncertainty. This suggests that the reduction in future scenarios is the main driver for the reduction in computation times for our \gls{2S-SP} and \gls{HA} approaches.

\begin{observation}
    Over the course of our operational planning problem under uncertainty, the subproblems of our \gls{2S-SP} and \gls{HA} approaches get easier to solve, while the subproblems of our Myopic approach get harder to solve. 
\end{observation}

\subsubsection{Computational effectiveness for varying numbers of scenarios}
In this section, we focus on how different numbers of scenarios in our \gls{2S-SP} and \gls{HA} approaches change the computational effectiveness of our approaches. Note that as our Myopic approach is independent of the number of scenarios, we report only one data point in the tables.

As expected, Table~\ref{tab: comparison computational effectiveness different scenarios 55} shows a nearly linear increase in serial computation times for our \gls{2S-SP} and \gls{HA} approaches for Route 55. We also expected nearly constant parallel computation times in our \gls{2S-SP} and \gls{HA} approaches over the different numbers of scenarios. This is the case for both our \gls{2S-SP} and \gls{HA} approaches. We also see that over all CSFs, even with a limited number of scenarios, our \gls{2S-SP} and \gls{HA} approaches improve the results of our Myopic approach by more than four percent. Even with an increase in scenarios, the solution quality of our \gls{2S-SP} and \gls{HA} approaches does not increase significantly.

\begin{table}[!b]
  \centering
  \caption{Comparison between \gls{2S-SP}, \gls{HA}, and Myopic with different number of future scenarios Route 55}
  \label{tab: comparison computational effectiveness different scenarios 55}
  \small
\begin{tabular}{cccccccccccc}
\multicolumn{1}{l}{} &  \multicolumn{4}{c}{\gls{2S-SP}} &  \multicolumn{4}{c}{\gls{HA}} &  \multicolumn{3}{c}{Myopic} \\
 \cmidrule(lr){2-5} \cmidrule(lr){6-9} \cmidrule(lr){10-12}
 Scenarios & Profit & SP & ser. & par. & Profit & SP & ser. & par. & Profit & SP & ser.\\
\midrule
5 & 22882.12 & 75.65 & 1.10 & 0.17 & 22793.50 & 75.37 & 0.73 & 0.18 \\
10 & 22905.57 & 75.74 & 2.32 & 0.19 & 22760.35 & 75.31 & 1.65 & 0.23 \\
20 & 22972.39 & 75.93 & 5.54 & 0.27 & 22778.48 & 75.25 & 3.19 & 0.27 & 21915.37 & 73.37 & 0.08 \\
30 & 22980.84 & 75.81 & 7.76 & 0.27 & 22828.88 & 75.49 & 4.59 & 0.28 \\
40 & 23021.10 & 75.79 & 12.10 & 0.32 & 22842.33 & 75.46 & 6.29 & 0.30 \\
\bottomrule
\end{tabular}
\end{table}

\begin{figure}[!b]
\centering
\vspace{-2ex}
	\begin{subfigure}[t]{.49\textwidth}
		\centering		
		\includegraphics[width=\textwidth]{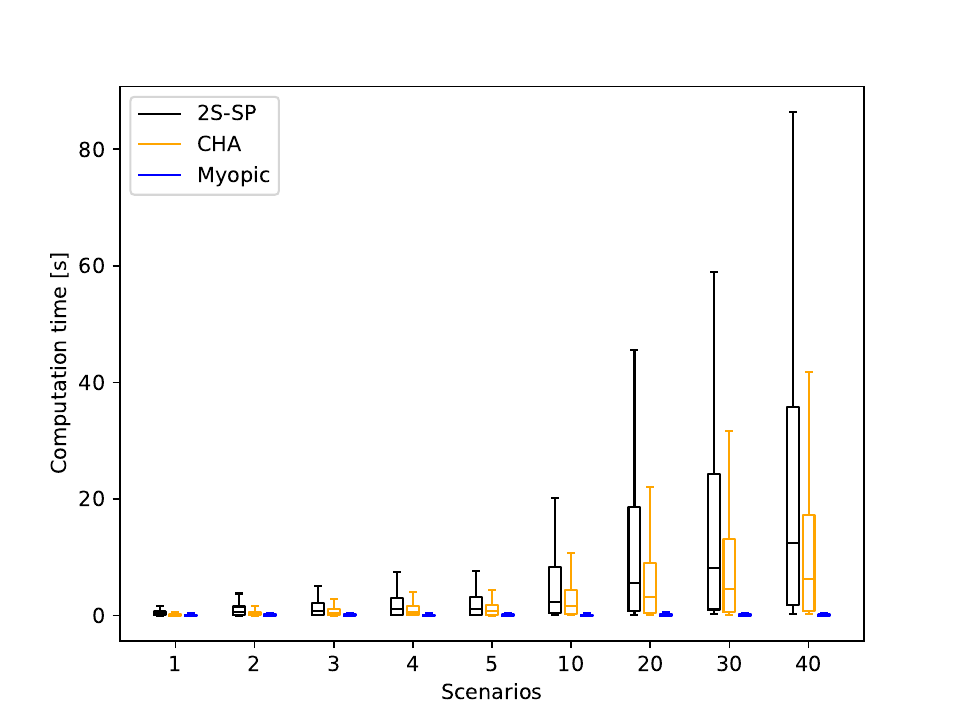}
        \captionsetup{format=hang}
		\caption{Serial computation times}
	\end{subfigure}
	\begin{subfigure}[t]{.49\textwidth}
		\centering		
		\includegraphics[width=\textwidth]{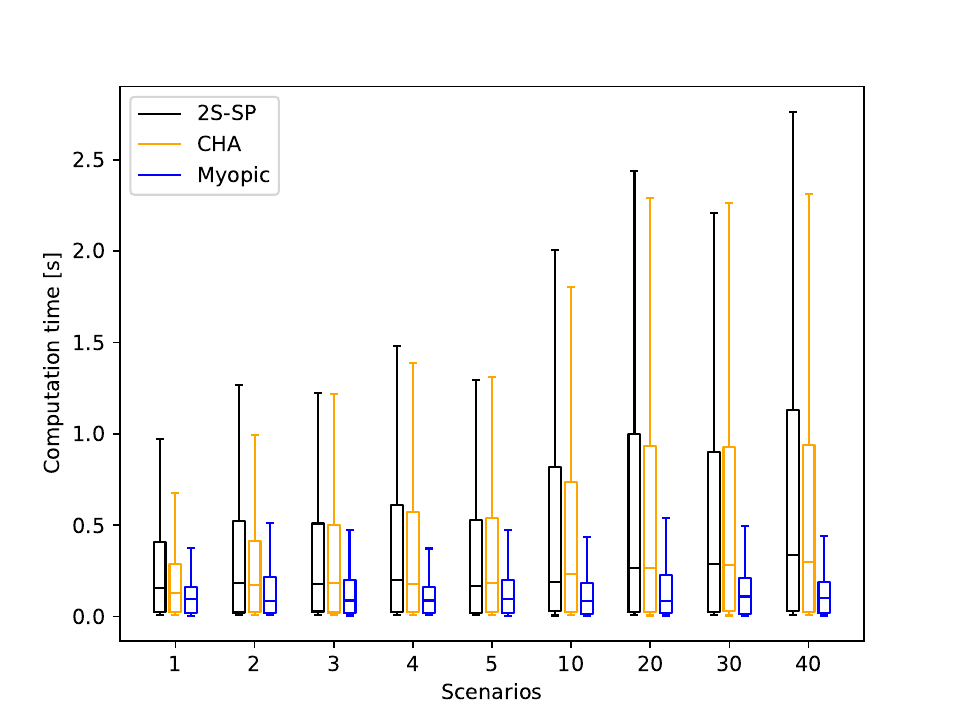}
        \captionsetup{format=hang}
		\caption{Parallel computation times}
	\end{subfigure}
 	\caption{Development of serial and parallel computation times based on different numbers of future scenarios for Route 55}
	\label{fig: number future scenarios}
\end{figure}

Figure~\ref{fig: number future scenarios} shows the distribution for the development of serial and parallel computation times for Route 55. Here, we see an increase in serial computation times and serial computation time variances when the number of scenarios increases, especially with our \gls{2S-SP} approach. As our Myopic approach is independent of the number of scenarios, its computation times remain constant. Apart from five scenarios, the parallel computation times and parallel computation time variances remain stable.  

Table~\ref{tab: comparison computational effectiveness different scenarios 155} confirms our observations of a nearly linear increase in serial computation times for our \gls{2S-SP} and GA approaches and nearly constant parallel computation times for Route 155. Additionally, we observe that a limited number of scenarios suffices for our \gls{2S-SP} and \gls{HA} approaches to determine better solutions than our Myopic approach. We elaborate on the quality of our operational planning problem policies under uncertainty in the next subsection, where we compare them to full information policies.

\begin{table}[!b]
  \centering
  \caption{Comparison between \gls{2S-SP}, \gls{HA}, and Myopic with different number of future scenarios Route 155}
  \label{tab: comparison computational effectiveness different scenarios 155}
  \small
\begin{tabular}{cccccccccccc}
\multicolumn{1}{l}{} &  \multicolumn{4}{c}{\gls{2S-SP}} &  \multicolumn{4}{c}{\gls{HA}} &  \multicolumn{3}{c}{Myopic} \\
 \cmidrule(lr){2-5} \cmidrule(lr){6-9} \cmidrule(lr){10-12}
 Scenarios & Profit & SP [\%] & ser. [s] & par. [s] & Profit & SP & ser. & par. & Profit & SP & ser.\\
\midrule
5 & 17283.31 & 82.47 & 0.55 & 0.06 & 17201.22 & 82.12 & 0.32 & 0.07 \\
10 & 17259.57 & 82.37 & 1.23 & 0.09 & 17237.89 & 82.22 & 0.76 & 0.10 \\
20 & 17267.28 & 82.42 & 2.71 & 0.10 & 17228.35 & 82.22 & 1.59 & 0.10 & 16852.43 & 81.01 & 0.04 \\
30 & 17264.25 & 82.47 & 4.22 & 0.12 & 17245.32 & 82.27 & 2.39 & 0.12 \\
40 & 17280.81 & 82.52 & 5.51 & 0.11 & 17260.91 & 82.32 & 3.40 & 0.13 \\
\bottomrule
\end{tabular}
\end{table}

\begin{figure}[!b]
\centering
\vspace{-2ex}
	\begin{subfigure}[t]{.49\textwidth}
		\centering		
		\includegraphics[width=\textwidth]{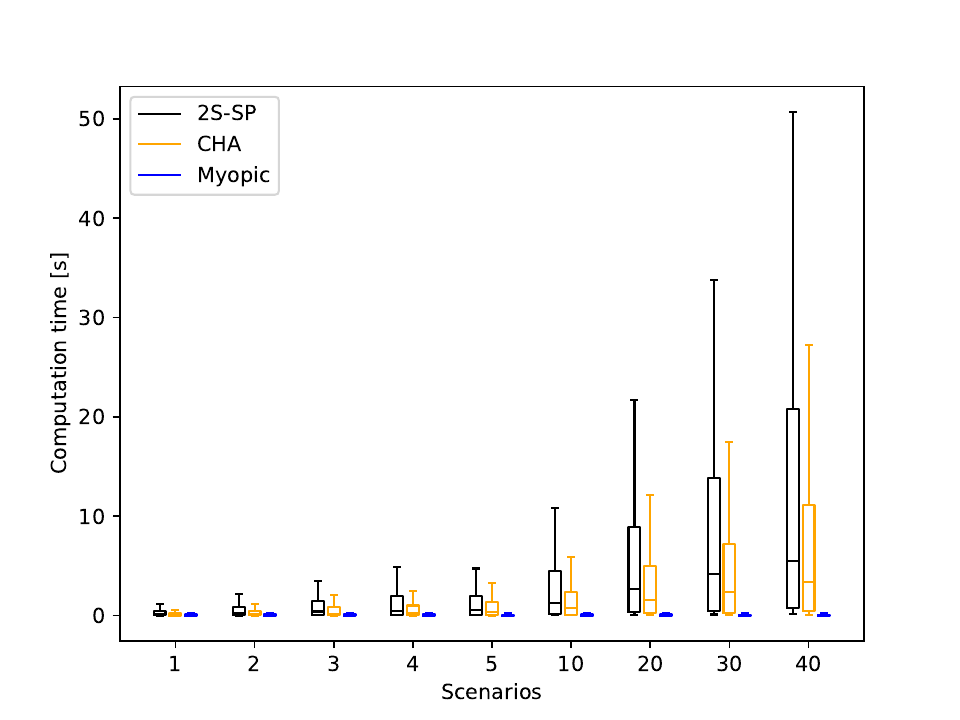}
        \captionsetup{format=hang}
		\caption{Serial computation times}
	\end{subfigure}
	\begin{subfigure}[t]{.49\textwidth}
		\centering		
		\includegraphics[width=\textwidth]{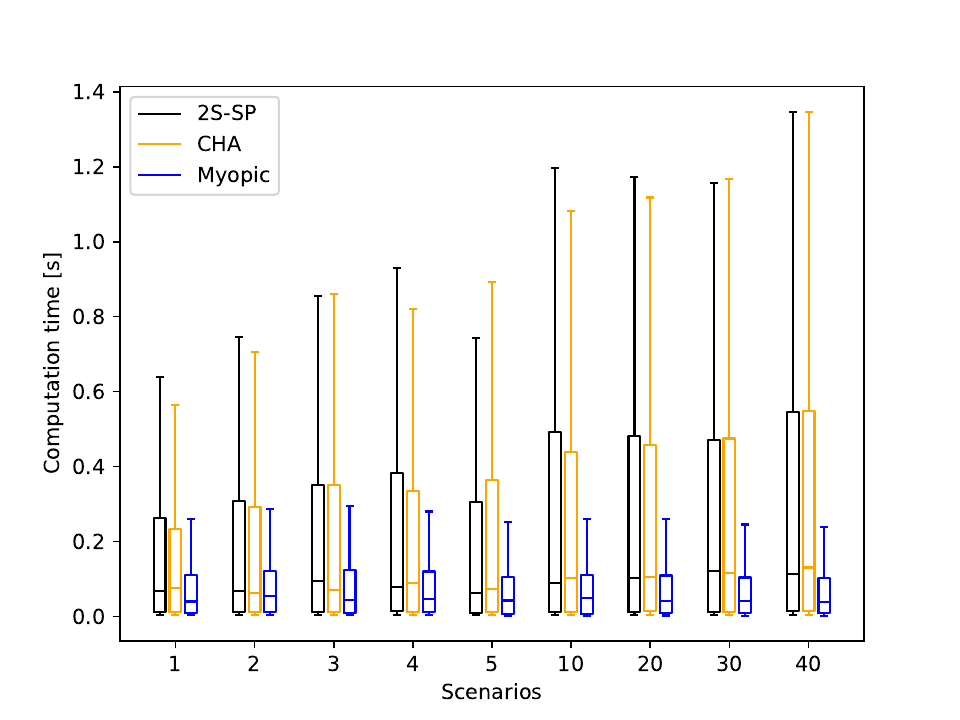}
        \captionsetup{format=hang}
		\caption{Parallel computation times}
	\end{subfigure}
 	\caption{Development of serial and parallel computation times based on different numbers of scenarios for Route 155}
	\label{fig: number future scenarios 155}
\end{figure}

Figure~\ref{fig: number future scenarios 155} shows that for Route 155, the serial computation time variance of our \gls{2S-SP} approach increase more with a growing number of scenarios than for our \gls{HA} approach. Contrary to our observations for Route 55, we see that for Route 155, the parallel computation times and parallel computation time variances in our \gls{2S-SP} approach slightly increase with a growing number of scenarios. Starting from 20 scenarios, the parallel computation times and parallel computation times variances stay nearly constant for our \gls{HA} approach. In both Routes 55 and 155, we see that the parallel computation times of our \gls{HA} approach are lower than our \gls{2S-SP} approach. This suggests that our subproblems are easier solvable without fixing the decision of a current request.

\begin{observation}
    The serial computation times of our \gls{2S-SP} and \gls{HA} approaches grow nearly linear with the number of scenarios. Additionally, the median parallel computation time stays constant with a growing number of scenarios. Note that this comes at the expense of additional memory usage. As our Myopic approach is independent of the number of scenarios, its serial computation times remain constant.
\end{observation}

\begin{observation}
    The variance in parallel computation time is lower for our \gls{HA} compared to our \gls{2S-SP} approach. Accordingly, not fixing the decision of whether to accept or reject a passenger request in our subproblems leads to more stable computation times.
\end{observation}

\subsubsection{Performance analysis of our approaches}
Our \gls{2S-SP} and \gls{HA} models use a set of scenarios to approximate the distribution of future passenger requests. In this context, one major question that arises is how many scenarios are necessary to approximate future demands so that we achieve a good approximation of future passenger demands and thus can maximize the operators' profit accordingly. To answer this question, we use our results from the previous section and compare them against the optimal full information policy obtained through Problem~\ref{prob: offline arcs}. Specifically, we compare our solutions obtained in a dynamic stochastic problem setting to the deterministic problem setting in which all requests are revealed at the same time.

Figure~\ref{fig: revenue scenarios} shows the optimality gap, i.e., \textit{1 - (objective of our policy under uncertainty / objective of full information policy)},  between our policies under uncertainty and optimal full information policies based on the number of scenarios for Routes 55 and 155 over all CSF configurations. Note that our Myopic approach is independent of the number of scenarios. We see that even with only five scenarios, both our \gls{2S-SP} and \gls{HA} approaches obtain good solutions. Additionally, we see that increasing the number of scenarios to more than five shows no significant improvements.

\begin{figure}[!b]
\centering
\vspace{-2ex}
	\begin{subfigure}[t]{.49\textwidth}
		\centering		
		\includegraphics[width=\textwidth]{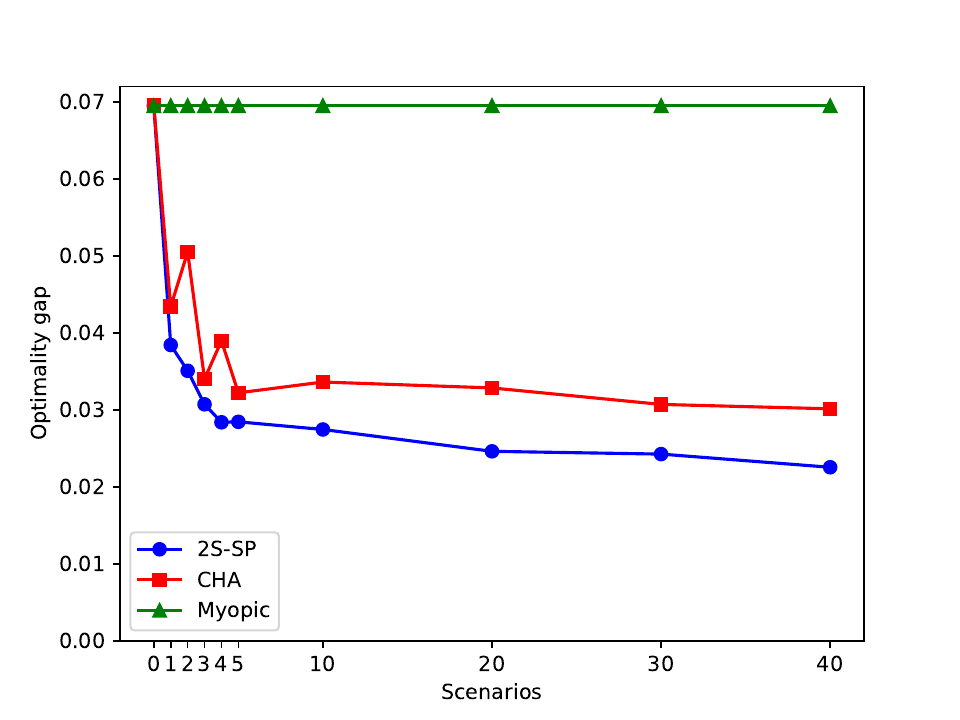}
        \captionsetup{format=hang}
		\caption{Optimality gap for Route 55}
	\end{subfigure}
	\begin{subfigure}[t]{.49\textwidth}
		\centering		
		\includegraphics[width=\textwidth]{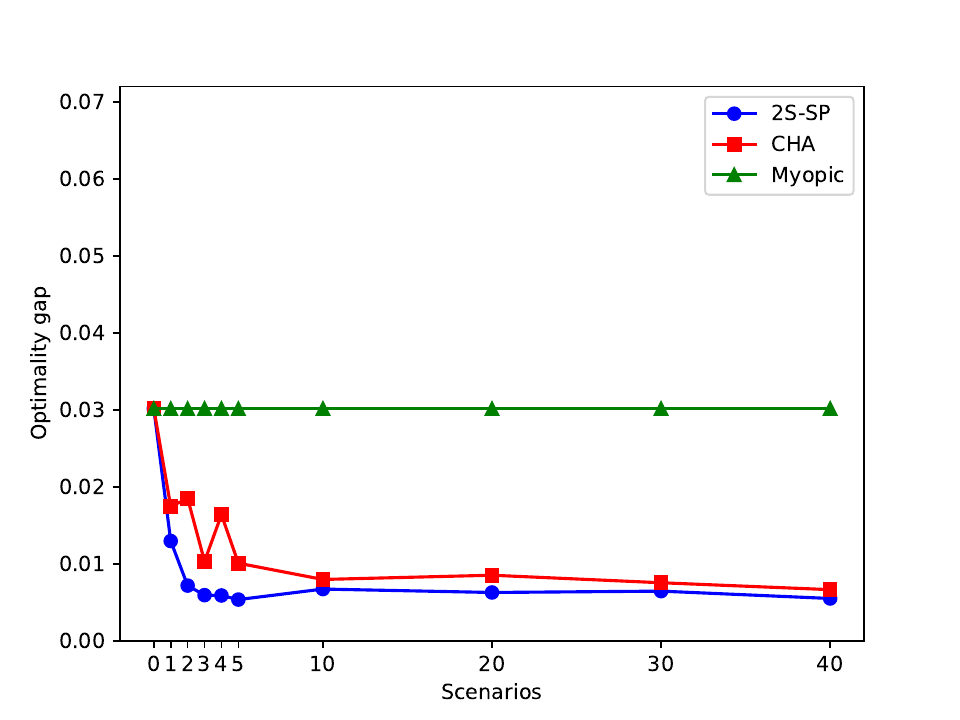}
        \captionsetup{format=hang}
		\caption{Optimality gap for Route 155}
	\end{subfigure}
 	\caption{Difference in profit based on the number of scenarios for Route 55 and 155}
	\label{fig: revenue scenarios}
\end{figure}

Figure~\ref{fig: revenue CSF} shows the difference in profit for varying CSFs for Routes 55 and 155 over all number of scenarios compared to the optimal full information policy. We see that especially for routes with higher flexibility, i.e., CSF of 0.2 and 0.4 our \gls{2S-SP} and \gls{HA} approaches perform much better than our Myopic approach. Additionally, we also see that for very low and very high route flexibility, i.e., CSF of 0.2, 0.8, and 1, especially our \gls{2S-SP} and \gls{HA} approaches are close to the optimal full information policy. The reason for that is that with very high flexibility, most of the passengers can be served, and with very low flexibility, the route is mostly fixed. Both cases are relatively easy to handle. As a \gls{DAS} combines flexibility with predictability, especially a CSF between $0.4 - 0.8$ is of interest. Here, our Myopic approach struggles to find good policies for Route 55. 

\begin{figure}[!t]
\centering
\vspace{-2ex}
	\begin{subfigure}[t]{.45\textwidth}
		\centering		
		\includegraphics[width=\textwidth]{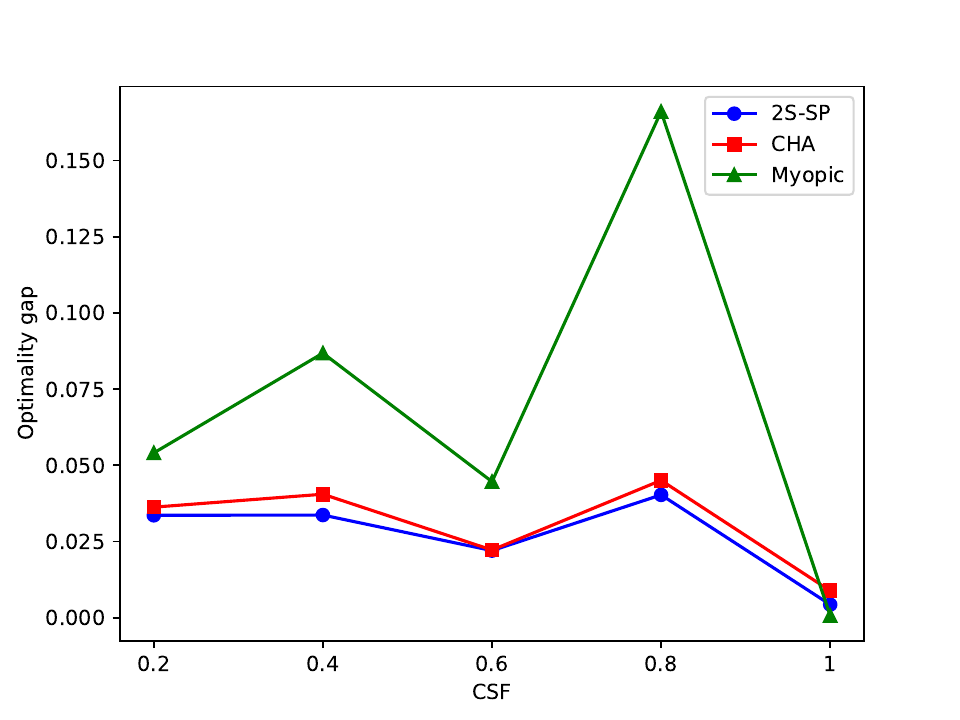}
        \captionsetup{format=hang}
		\caption{Profit for Route 55}
	\end{subfigure}
	\begin{subfigure}[t]{.45\textwidth}
		\centering		
		\includegraphics[width=\textwidth]{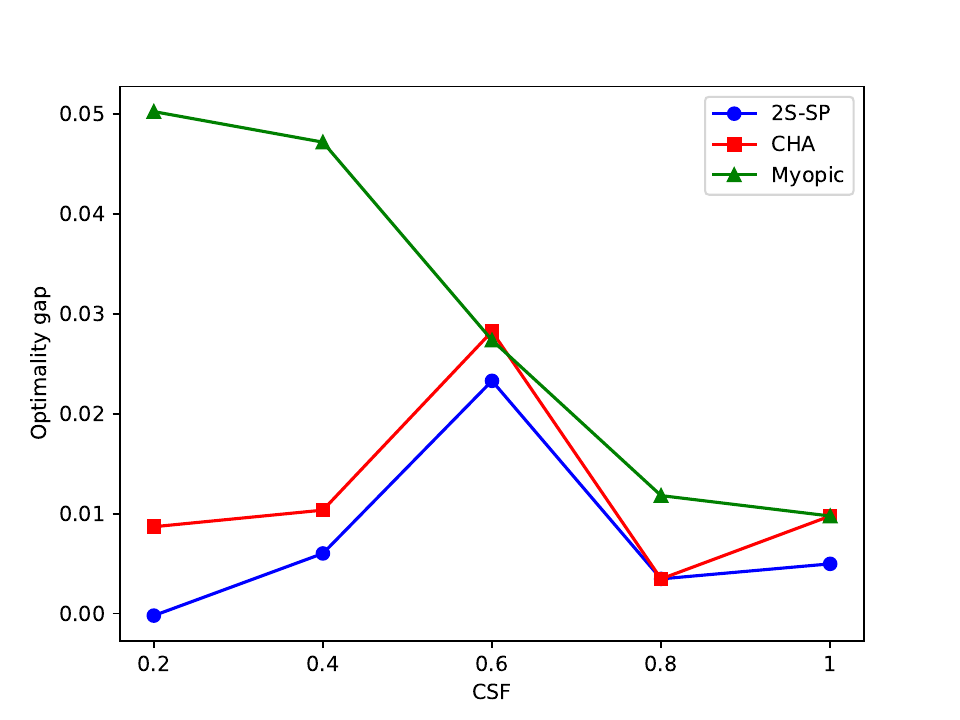}
        \captionsetup{format=hang}
		\caption{Profit for Route 155}
	\end{subfigure}
  	\caption{Difference in profit based on varying CSFs for Route 55 and 155}
	\label{fig: revenue CSF}
\end{figure}

\begin{observation}
    Even with a limited amount of scenarios, our \gls{2S-SP} and \gls{HA} approaches can determine good policies under uncertainty. Especially in high and low flexibility instances, our \gls{2S-SP} and \gls{HA} approaches perform better than our Myopic approach. Additionally, with decreasing flexibility, all our approaches can determine close to optimal full information policies. 
\end{observation}

\subsection{Practical insights}
From the perspective of a public transportation system operator, a natural question is whether it is beneficial to convert a fixed-line bus route to a \gls{DAS}. To answer this question, our analysis in this section is twofold. First, we compare the average profits, percentages of passengers served, and travel times between the fixed-line bus routes, a \gls{DAS} under uncertainty, and a \gls{DAS} in the full information setting. Second, we analyze the impact of passenger willingness to walk and route flexibility on profit, served passengers, and the cost per passenger.

\subsubsection{Comparison between fixed-line and a \gls{DAS}}
One of our main motivations for the operational planning problem under uncertainty is to be able to compare fixed-line public bus transportation systems with a \gls{DAS} on an operational level. To do so, we solve Problem~\ref{prob: offline arcs} with fixed arcs between the compulsory stops for a CSF of one, on five different instances. This gives us average profits, percentages of passengers served, and travel times for a fixed-line bus transportation system. Afterward, we determine the same metrics for the full information \gls{DAS} setting through Problem~\ref{prob: offline arcs} without fixed arcs and for the \gls{DAS} setting under uncertainty through Problem~\ref{prob: decomp Q}. In all three settings, the walking distance is 250 meters, the utility is 750, and in the two \gls{DAS} settings, the CSF $\in [0.2, 0.4, 0.6, 0.8, 1]$. Note that we only consider the fixed-line bus setting over the compulsory stops of the original routes, i.e., only for a CSF of one.

\begin{figure}[!b]
 	\begin{subfigure}[t]{.49\textwidth}
		\centering		
		\includegraphics[width=\textwidth]{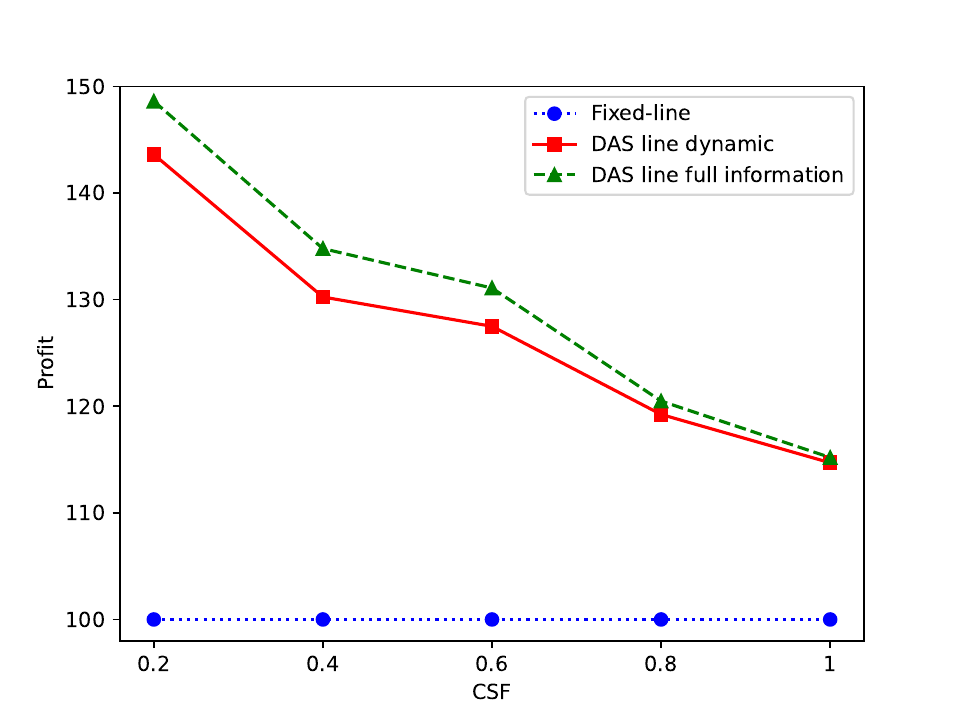}
        \captionsetup{format=hang}
		\caption{\centering Difference in profits between fixed-line and a \gls{DAS} settings}
	\end{subfigure}
	\begin{subfigure}[t]{.49\textwidth}
		\centering		
		\includegraphics[width=\textwidth]{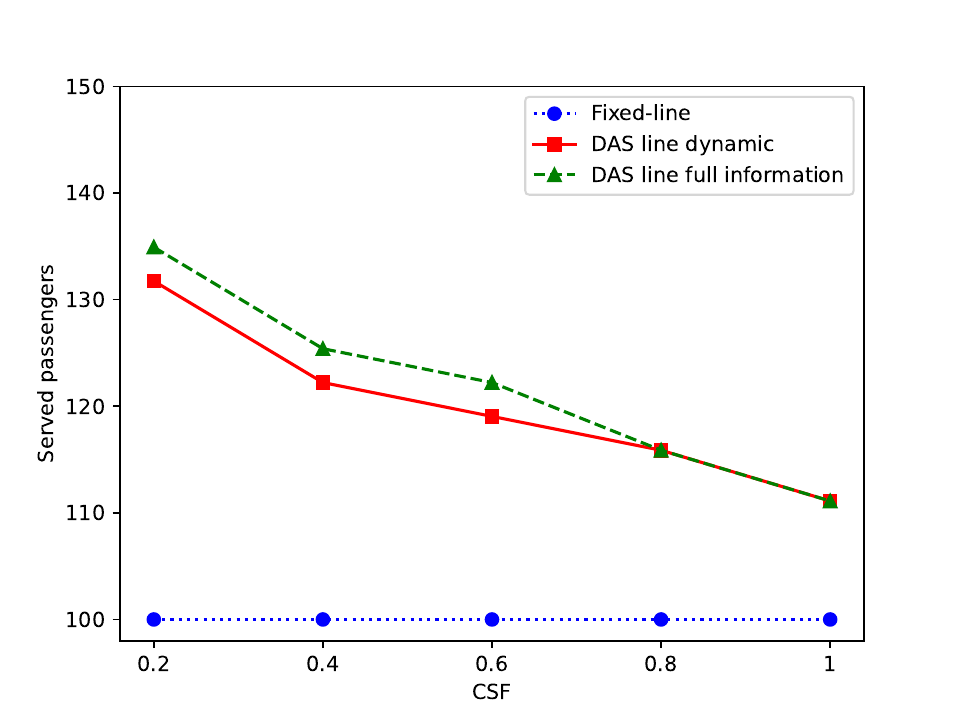}
        \captionsetup{format=hang}
		\caption{\centering Difference in served passengers between fixed-line and a \gls{DAS} settings}
	\end{subfigure}

  \begin{subfigure}[t]{.49\textwidth}
		\centering		
		\includegraphics[width=\textwidth]{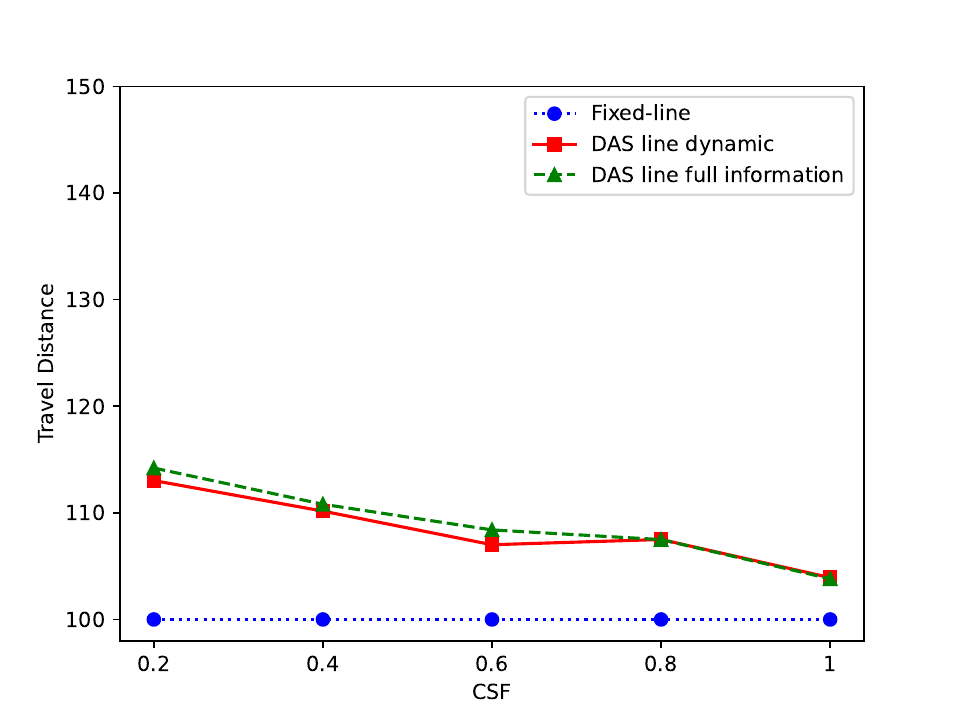}
        \captionsetup{format=hang}
		\caption{\centering Difference in travel times between fixed-line and a \gls{DAS} settings}
	\end{subfigure}	
 \caption{Increase in profit, served passengers, and travel times compared to a fixed-line baseline setting as 100 for Route 55}
 \label{fig: comparison fixed-line and DAS metrics}
\end{figure}

Figure~\ref{fig: comparison fixed-line and DAS metrics} shows the comparison between fixed-line and \gls{DAS} settings. Here, the metrics of the solution of the fixed-line setting are scaled to the baseline of 100, e.g., a profit of 115 for the dynamic \gls{DAS} setting constitutes a 15\% increase in profit compared to the fixed-line setting.
We see that even for a CSF of one, both \gls{DAS} settings generate 15\% more profit, serve 11\% more passengers, while having only 4\% more travel time. With growing flexibility, i.e., a decrease in CSFs, these metrics increase to 44 - 49\% more profit, serving 32 - 35\% more passengers, with an increase of 13 - 14\% in travel time. Over all CSF, we see that in our route, an increase of one percent in travel distance leads to an increase of 2.7 - 4.7\% in profits and 2.2 - 3.7\% in passengers served in our setting under uncertainty. Additionally, an increase of one percent in travel distance leads to an increase of 2.9 - 5\% in profits and 2.3 - 3.7\% in passengers served in our full information setting. An analysis of Route 155 yields similar results.

\begin{observation}
    Switching from a fixed-line setting to a \gls{DAS} increases profits and passengers served at the cost of longer travel distances of the bus and thus its passengers. With a travel distance increase of one percent, a \gls{DAS} can achieve up to 4.7\% more profit, serving up to 3.7\% more passengers in a setting under uncertainty. 
\end{observation}

\subsubsection{Impact of passenger walking willingness and route flexibility}

Our model allows us to analyze the number of passengers served and costs per served passenger based on different \gls{DAS} assumptions, i.e., different CSFs, which indicates the degree of flexibility and the walking distance passengers are willing to walk to and from stops. Our analysis in this section is twofold. First, we analyze the impact of those two parameters on the number of passengers served and achievable profit. Second, we investigate how CSF and walking distance influence the cost per passenger served. To do so, we compare the solutions of our exact approach with five scenarios for a uniform passenger utility of 750, a CSF $\in [0.2, 0.4, 0.6, 0.8, 1]$, and a walking distance~$\in~[100, 150, 250, 300, 350, 400]$ on Route 55 and 155.

\subsubsection*{Profit and passengers served}
Our model presumes a CSF indicating the flexibility of the \gls{DAS} line and a threshold distance passengers are willing to walk to nearby stops. To understand which parameter drives profit and the number of served passengers, we aggregate results on instances according to the previously mentioned parameter values for the CSF and the walking distance.

\begin{figure}[!t]
\centering
\begin{subfigure}[t]{.45\textwidth}
	\centering
 \includegraphics[width=1.1\textwidth]{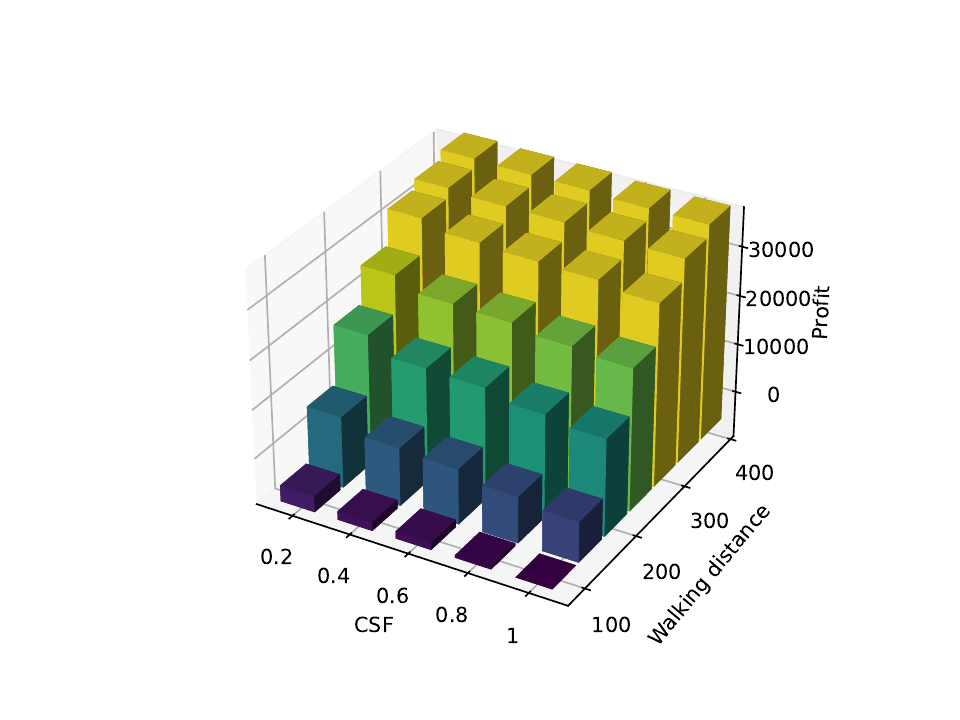}
    \captionsetup{format=hang}
	\caption{Average operator profit}
		\label{fig: revenue development}
\end{subfigure}
\begin{subfigure}[t]{.45\textwidth}
	\centering
 \includegraphics[width=1.1\textwidth]{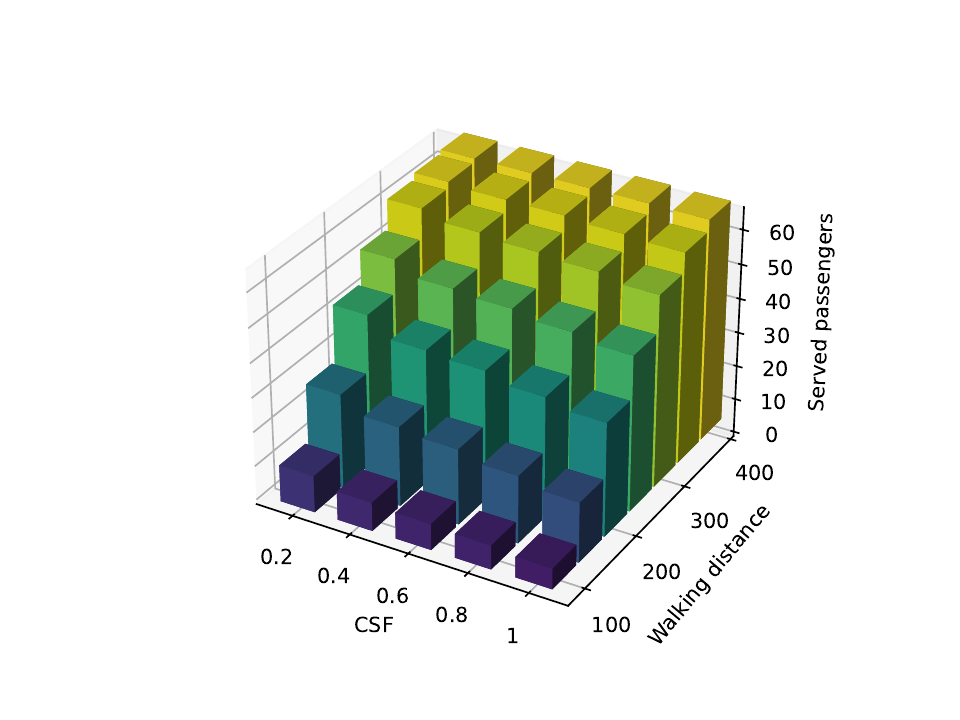}
    \captionsetup{format=hang}
	\caption{Average number of served passengers}
		\label{fig: served passengers development}
\end{subfigure}
\caption{Development of operator profit and number of served passengers based on route flexibility (CSF) and walking distance for Route 55}
\label{fig: walking distance and compulsory stop factor impact route 55}
\end{figure}

    Figure~\ref{fig: walking distance and compulsory stop factor impact route 55} shows that for Route 55, the profit increases from -8,549 to 36,961 with increasing route flexibility and passenger walking willingness, and the number of served passengers increases from 6 to 65.1. Here, walking distances larger than 300 meters allow the operator to serve between 86-100\% of all passenger requests. 
    We see that especially for walking distances up to 250 meters, decreasing the CSF from 1 to 0.2 increases the profit by 59\% from 11,453 to 18,221 and the number of served passengers by 19\% from 45.4 to 53.8. At the same time, increasing the walking distance from 100 to 250 meters increases the profit from -7,133 to 22,906, and passengers served by 620\% from 7.9 to 49.1. This shows that even though a higher degree of flexibility allows an operator to serve more passengers and thus increase profit, incentivizing passengers to walk slightly larger distances to and from their stops represents a larger lever to increase profit.

\begin{figure}[!t]
\centering
\begin{subfigure}[t]{.45\textwidth}
	\centering
 \includegraphics[width=1.1\textwidth]{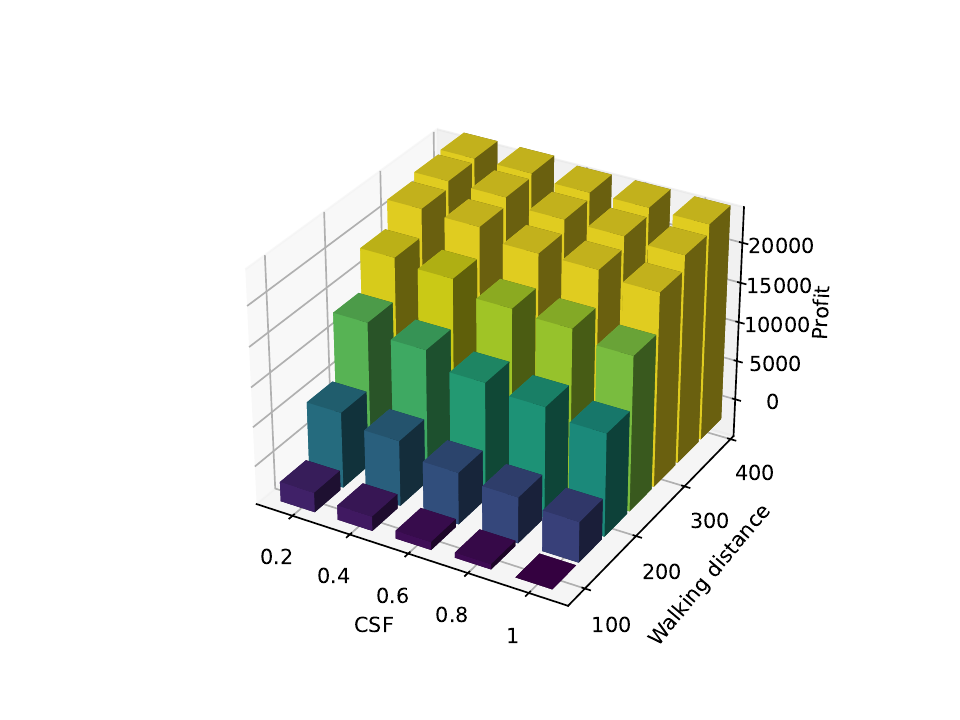}
    \captionsetup{format=hang}
	\caption{Average operator profit}
		\label{fig: revenue development}
\end{subfigure}
\begin{subfigure}[t]{.45\textwidth}
	\centering
 \includegraphics[width=1.1\textwidth]{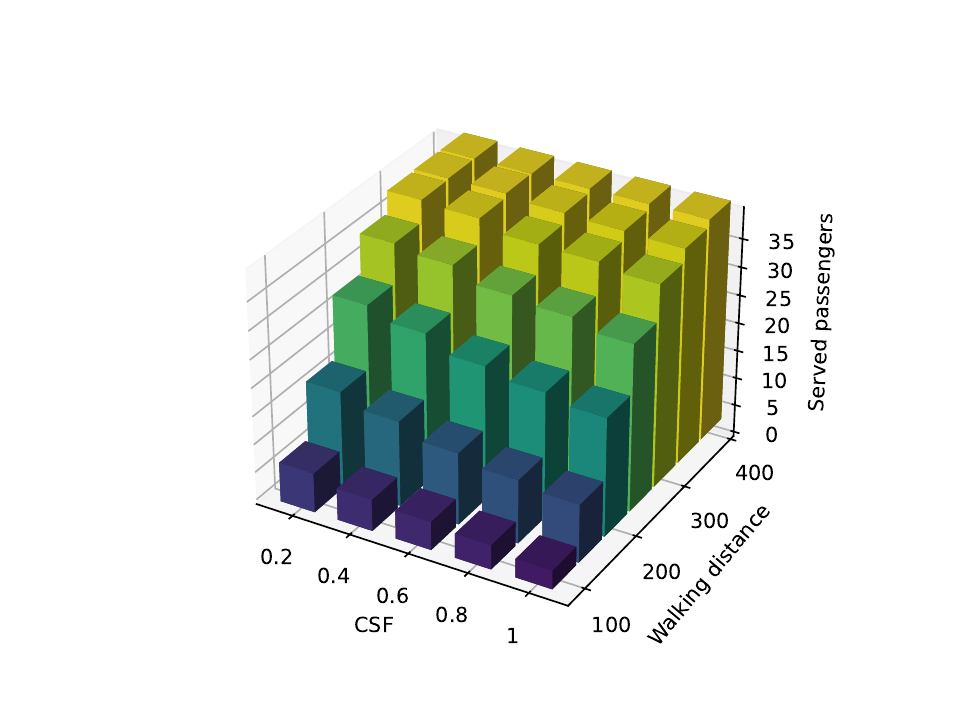}
    \captionsetup{format=hang}
	\caption{Average number of served passengers}
		\label{fig: served passengers development}
\end{subfigure}
\caption{Development of operator profit and number of served passengers based on route flexibility (CSF) and walking distance for Route 155}
\label{fig: walking distance and compulsory stop factor impact route 155}
\end{figure}

    Figure~\ref{fig: walking distance and compulsory stop factor impact route 155} shows that our previous observations also holds for Route 155.
    We see that with increasing route flexibility and passenger walking willingness, the profit increases from -4,283 to 23,812, and the number of served passengers increases from 3.3 to 39.7. Here, walking distances larger than 300 meters allow the operator to serve between 91-100\% of all issued passenger requests. 
    We also see that for shorter walking distances up to 250 meters, decreasing the CSF from 1 to 0.2 increases the profit by 25\% from 15,283 to 19,127 and the number of served passengers by 24\% from 28.6 to 35.6. At the same time, increasing the walking distance from 100 to 250 meters increases the profit from -3,093 to 17,230, and passengers served by 654\% from 5.0 to 32.7.
    
\begin{observation}
    With increasing route flexibility and a higher willingness to walk, operators can serve more passengers and generate higher profits. Solely increasing route flexibility does not suffice to significantly improve the number of served passengers. To achieve this, incentivizing passengers to walk slightly longer distances to and from their stops, e.g., through dynamic pricing, is necessary. 
\end{observation}

\subsubsection*{Cost per passenger}
Apart from operator profit and the number of served passengers, we additionally want to analyze the development of the cost per passenger based on varying passenger walking willingness and route flexibility. To do so, we define the cost per passenger as the routing cost of the vehicle divided by the number of served passengers. 

\begin{figure}[!b]
\centering
\begin{subfigure}[t]{.45\textwidth}
	\centering
 \includegraphics[width=1.1\textwidth]{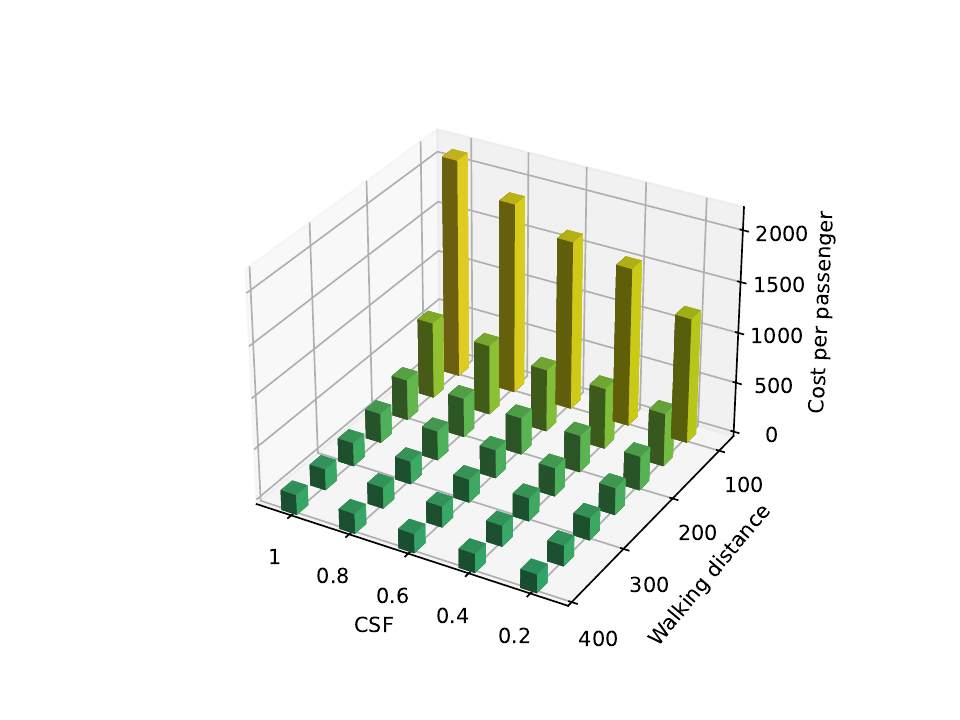}
    \captionsetup{format=hang}
	\caption{Route 55}
		\label{fig: revenue development}
\end{subfigure}
\begin{subfigure}[t]{.45\textwidth}
	\centering
 \includegraphics[width=1.1\textwidth]{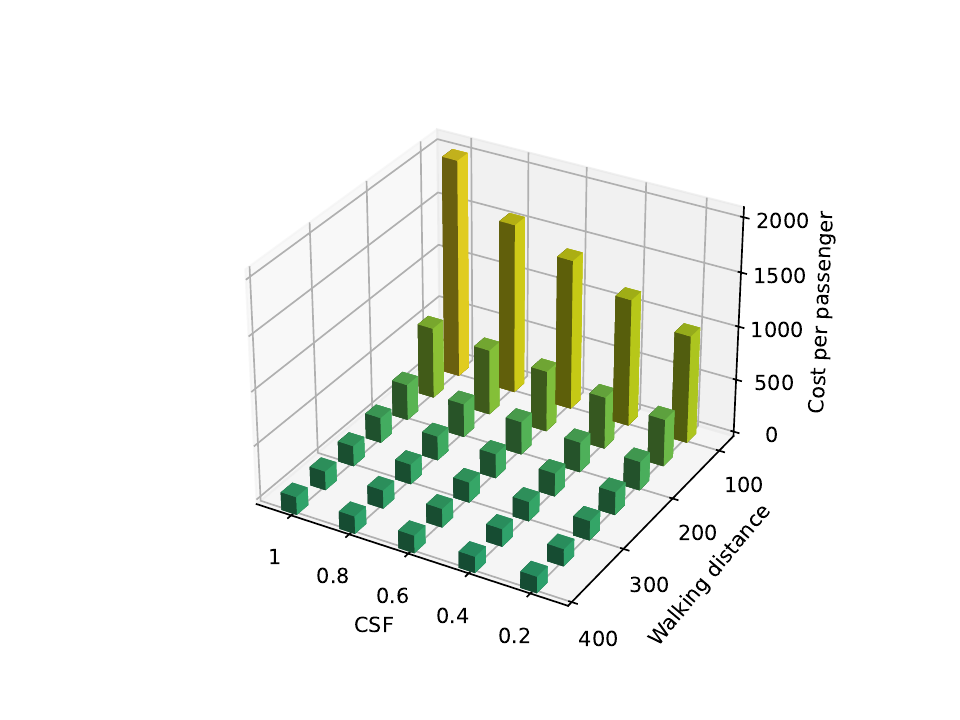}
    \captionsetup{format=hang}
	\caption{Route 155}
		\label{fig: served passengers development}
\end{subfigure}
\caption{Development of cost per passenger based on route flexibility (CSF) and walking distance}
\label{fig: walking distance and compulsory stop factor cost per passenger}
\end{figure}

Figure~\ref{fig: walking distance and compulsory stop factor cost per passenger} shows that by increasing the passenger walking distance and the flexibility of the route, the cost per passenger for the two routes reduces from 2,175 to 182 and 2,048 to 150, respectively. We see that even though the flexibility of the route reduces the cost per passenger between 3-43\% and 5-51\%, especially for lower walking distances under 250 meters, the passengers' willingness to walk has an even higher impact on the cost per passenger. Here, we see that increasing the walking distance from 100 to 250 meters reduces the cost per passenger for the two routes on average by 83\% from 1,712 to 284 and by 85\% from 1,449 to 223, respectively. Generally, over the different CSF configurations, the cost per passenger in both routes reduces between 88-90\% for a walking distance increase from 100 to 400 meters. 

\begin{observation}
If passengers are willing to walk 250 meters instead of 100 meters to and from stops, the cost per passenger reduces by 83-85\%. Increasing the flexibility of a route by reducing the CSF from 1 to 0.2 only reduces the cost per passenger by 43-51\%.

\end{observation}

\section{Conclusion and future work}
\label{sec: Conclusion}
We introduced the operational planning problem under uncertainty of a \acrlong{DAS}. Furthermore, we proposed an algorithmic framework that allows an operator to plan in real-time which passengers to serve in a \acrlong{DAS} line. 
To this end, we formulated the operational route planning problem's full information version as a mixed-integer program, modeled the operational route planning problem under uncertainty as a \acrlong{MDP}, and utilized a rolling horizon approach to compute a policy via a two-stage stochastic program in each timestep to decide whether to accept or reject a passenger. Furthermore, we determined the deterministic equivalent of our approximation through sample-based approximation. This allowed us to decompose the deterministic equivalent of our two-stage stochastic program into several full information planning problems that can be solved in parallel efficiently. Additionally, we proposed a heuristic and a myopic approach. 

We performed extensive computational experimentation to analyze how our different algorithms perform on different routes, the number of future passenger demand scenarios, passenger walking willingness, and route flexibility. 
For this purpose, we generated a large set of instances based on real-world data provided to us by Stadtwerke München GmbH, a Munich public transportation provider. 
In a comparison between full information policies and policies under uncertainty, we showed that in our case study, a limited amount of future request scenarios suffices to obtain good policies under uncertainty, especially in routes with very low and high flexibility. 

From a managerial perspective, we show that by switching a fixed-line bus route to a \acrlong{DAS}, an operator can increase revenue by up to 49\% and the number of served passengers by up to 35\% while only increasing the travel distance of the bus by 14\%.
Furthermore, we showed that even though we can achieve an increase in passengers served between 19-24\% by making a route more flexible, increasing passengers' willingness to walk from 100 to 250 meters increases this number by 620-654\%. Additionally, such an increase in passenger willingness to walk also reduced the cost per passenger served by 83-85\%, while an increase in the flexibility of the route only reduced it by 43-51\%.

Furthermore, we showed that our exact decomposition consistently derives the best solutions in terms of operator revenues and served passengers, followed by our heuristic approach and, lastly, our myopic approach. This performance comes at the cost of higher serial computation time and doubles the resources needed in parallel computation compared to our heuristic approach, which is 17-57\% faster. Additionally, we saw an increase in problem complexity, i.e., serial computation time with growing route flexibility. 

Future research may expand on the current methodology of our two-stage stochastic program by developing a machine-learning enhanced policy that reduces the computational effort during operations. Additionally, our current passenger distribution is based on the assumption that passengers utilizing a \acrlong{DAS} have a similar demand distribution as passengers utilizing fixed-line bus systems. Investigating the impact of different passenger request distributions on the operators' revenue remains an interesting avenue for future research. 
Finally, considering passenger behavior, e.g., dynamic pricing models that influence passenger acceptance or rejection decisions, poses an interesting area for future study.

\section*{Acknowledgment}
This work was supported by the German Federal Ministry of Education and Research under Grant 03ZU1105FA. Additionally, we thank Stadtwerke München GmbH for the pleasant collaboration and for providing us with the necessary data.


%
\singlespacing{
\bibliographystyle{model5-names}
\bibliography{ms}} 
\newpage
\onehalfspacing
\begin{appendices}
	\normalsize
\end{appendices}
\end{document}